\documentclass[12pt]{article}
\usepackage{amsmath,amsbsy,amssymb,latexsym,amsthm}
\usepackage[linkcolor=blue,colorlinks=true]{hyperref}
\usepackage[numbers,sort&compress]{natbib}
\usepackage{cleveref}
\usepackage{indentfirst}
\usepackage{float}
\providecommand{\keywords}[1]{\textbf{\textit{Keyword:}} #1}
\providecommand{\msc}[1]{\textbf{\textit{Mathematics Subject Classifications (2000):}} #1}
\newtheorem{Definition}{Definition}[section]

\newtheorem{Theorem}{Theorem}[section]
\newtheorem{Remark}{Remark}[section]
\newtheorem{Example}{Example}[section]
\newtheorem{Lemma}{Lemma}[section]

\usepackage[edges]{forest}
\usepackage{caption}

\tikzset{parent/.style={align=center,text width=3cm,rounded corners=3pt},
	child/.style={align=center,text width=3cm,rounded corners=3pt}
}

\colorlet{col1}{white}
\colorlet{col2}{gray!20}
\colorlet{col3}{gray!40}
\colorlet{col4}{gray!60}
\newcommand{\footremember}[2]{%
	\footnote{#2}
	\newcounter{#1}
	\setcounter{#1}{\value{footnote}}%
}

\title{The sine and cosine diffusive representations for the Caputo fractional derivative }
\author{%
	Hassan Khosravian-Arab\footremember{alley}{ Department of Applied Mathematics and Computer Science, Faculty
		of Mathematics and Statistics, University of Isfahan, 81746-73441 Isfahan, Iran. {\it E-mail	Address:} h.khosravian@sci.ui.ac.ir.}%
	\and Mehdi Dehghan\footremember{trailer}{Corresponding author. \newline \indent\indent Department of Applied Mathematics, Faculty of Mathematics and Computer Science, Amirkabir University of Technology (Tehran Polytechnic),
		No.424, Hafez Avenue, Tehran, Iran. {\it E-mail	Address:} mdehghan@aut.ac.ir, mdehghan.aut@gmail.com.}%
}
\begin{document}
\date{}
\maketitle
\begin{abstract}
	As we are aware, various types of methods  have been proposed to approximate the Caputo fractional derivative numerically. A common challenge of the methods is  the non-local property of the Caputo fractional derivative which leads to the  slow and memory consuming methods. Diffusive representation of fractional derivative is an efficient tool to overcome the mentioned challenge. This paper presents two new diffusive representations to approximate the Caputo fractional derivative of order $0<\alpha<1$. Error analysis of the newly presented methods together with some numerical examples are provided at the end.
\end{abstract}

\keywords 
Caputo fractional derivative; Non-locality property; Diffusive representation; Infinite state representation; Memory free formulation;   Error analysis.

\msc 
26A33;
\ 65D30;
\ 65D25;
\ 65D32.
\section{Introduction}
Nowadays, there is an international awareness on the importance of fractional calculus as well as their broad applications
in various areas such as: mathematics, statistics, physics, chemistry, electronic, engineering, biology and etc \cite{Ji2017,Lei2017,Castillo2018,Du2020,MR2218073,MR4030088}. This means that many real world problems have been modeled by the following  fractional differential equation (FDEs):
\begin{equation}\label{Eq_1}
		{}^{C}D_{a^+}^{\alpha}y(t)=F(t,y(t)),\  y(a)=y_a,\ 0<\alpha\leq1,
\end{equation}
where ${}^{C}D_{a^+}^{\alpha}$ used for the Caputo fractional derivative of order $\alpha$ with starting point $a$ \cite{MR2680847,MR2218073}:
\begin{equation}\label{Eq_2}
		{}^{C}D_{a^+}^{\alpha}y(t)=\frac{1}{\Gamma(1-\alpha)}\int_{a}^{t}(t-\tau)^{-\alpha}y'(\tau)\,d\tau,\ 0<\alpha<1,
\end{equation}
where $\Gamma(.)$ is the Euler's Gamma function.

The first and most important step to solve FDEs \eqref{Eq_1} numerically, is to approximate the Caputo fractional derivative(s) ${}^{C}D_{a^+}^{\alpha}y(t)$. Unfortunately, due to the non-locality property of the Caputo fractional derivative, there is a significant computational challenge to approximate this operator numerically \cite{Diethelm2021,Diethelm2022}.  

The review of existing literature on the numerical solutions of FDEs \eqref{Eq_1} shows that, all the methods are based on the following two ideas \cite{Li2015}:
\begin{itemize}
	\item \textbf{Direct Methods}: In these methods the Caputo fractional derivative can be approximated directly to obtain the numerical schemes.
	\item \textbf{Indirect Methods}: In these methods,  the original problem \eqref{Eq_1} is transformed into the fractional integral
	equation and then using a suitable method to discretize the	fractional integral, the numerical schemes can be obtained.
\end{itemize}

The Direct Methods (DM) can be also divided into two main categories:
\begin{itemize}
	\item Nodal Methods.
	\item Modal Methods.
\end{itemize}

For the readers convenience, we summarized the existing DM to approximate Caputo fractional derivative \eqref{Eq_2}  in Fig. \ref{DM}.  

The main drawback of these methods is that to handle the non-locality of the fractional differential operators, they require a relatively large amount of time and/or computer memory \cite{Diethelm2022,Diethelm2021}.

Specially, when we approximate the Caputo fractional derivative of order $\alpha$ at points $\{x_j\}_{j=0}^n$ based on the nodal methods,  the most common difficulty which we face with is that the computational complexity of these methods is proportional to $\mathcal{O}(n^2)$ (for the classical convolution types), $\mathcal{O}(n \log^2 n)$ (for some modification types)  or $\mathcal{O}(n \log n)$ (by the use of the fast Fourier transform) and then for large values of $n$, the computational complexity of these methods increases very fast.

To overcome this drawback, a new representation of the Caputo fractional derivative (which is so-called as diffusive representation (DR), infinite state representation (ISR) or memory free formulation (MFF)) was introduced by Yuan and Agrawal in \cite{Yuan2002,Agrawal2009}.  In fact, they have shown that the Caputo fractional derivative can be reformulated as:
\begin{equation}\label{YARep_1}
	{}^{C}D_{0^+}^{\alpha}y(t)=\int_{0}^{+\infty}\phi(\omega,t)\,d\omega,
\end{equation}
where $\phi(\omega,t)$ for $\omega\in(0,+\infty)$ called the observed
system's infinite states at time $t$ and also satisfies the inhomogeneous first order differential equation in the following form:
\begin{equation}\label{YARep_Diff}
	\frac{\partial}{\partial t}\phi(\omega,t)=h_1(\omega)\phi(\omega,t)+h_2(\omega)y'(t),\ \phi(\omega,0)=0,
\end{equation}
with certain functions $h_1; h_2: (0,+\infty) \to \mathbb{R}$. Some new improvements and modifications of the diffusive representation have been introduced in \cite{Liu2018,MR3936246,Trinks2002,Schmidt2006,Diethelm2022,Hinze2019,MR4392021,MR3474912,Birk2010,Yuan2002,Audounet,Matignon2009,Agrawal2009}. One of the most important features of these methods is to reduce the computational complexity of fractional differential solvers effectively to just $\mathcal{O}(n)$ (See \cite{Diethelm2021,Ford2001,Garrappa2018,AnewDiff}). 
\begin{center}
	\resizebox*{.75\linewidth}{!}{%
		\begin{forest}
			forked edges,
			for tree={
				grow'=east,
				draw,
				rounded corners,
				text width=5cm,
				node options={align=center},
			}       
			[Direct Methods \cite{Li2015},text width=5cm, fill=col1, parent, s sep=2cm
			[Nodal Methods,text width=4cm, for tree={child, fill=col1,text width=4cm}
			[Polynomial Interpolation, for tree={child, fill=col1,text width=6cm}
			[Based on Gauss Nodes]
			[Based on Equispaced Nodes]
			[Based on Non-Equispaced Nodes]
			]
			[Spline Interpolation,text width=6cm]
			[Convolution Quadrature Rules,text width=6cm]
			]
			[Modal Methods, for tree={child, fill=col1,text width=4cm}
			[Based on the (Non)-Classical Polynomials,for tree={child, fill=col1,text width=5cm}
			[Classical and Fractional Taylor Series]
			[Bernstein and Other Polynomials ]
			]
			[Based on the Classical Orthogonal Polynomials,text width=5cm,for tree={child, fill=col1,text width=5cm}
			[Classical Jacobi Polynomials and Their Especial Cases]
			[Classical Laguerre Polynomials]
			[Classical Hermite Polynomials]
			]
			[Based on the Non-Classical Orthogonal Functions,text width=5cm,for tree={child, fill=col1,text width=5cm}
			[M\"{u}ntz Functions and Their Especial Cases  \cite{MR4276333,MR4162334}]
			[Jacobi Polyfractonomials and Their Especial Cases \cite{MR3101519,MR3465404}]
			[Generalized Laguerre Functions of the first and second type \cite{MR3384739,MR3628263}]
			]
			[Based on the Classical Piece-wise Orthogonal Polynomials,text width=5cm]
			[Based on the Non-Classical Piece-wise Orthogonal Polynomials,text width=5cm]
			]
			]
			]
	\end{forest}}
	\captionsetup{type=figure}
	
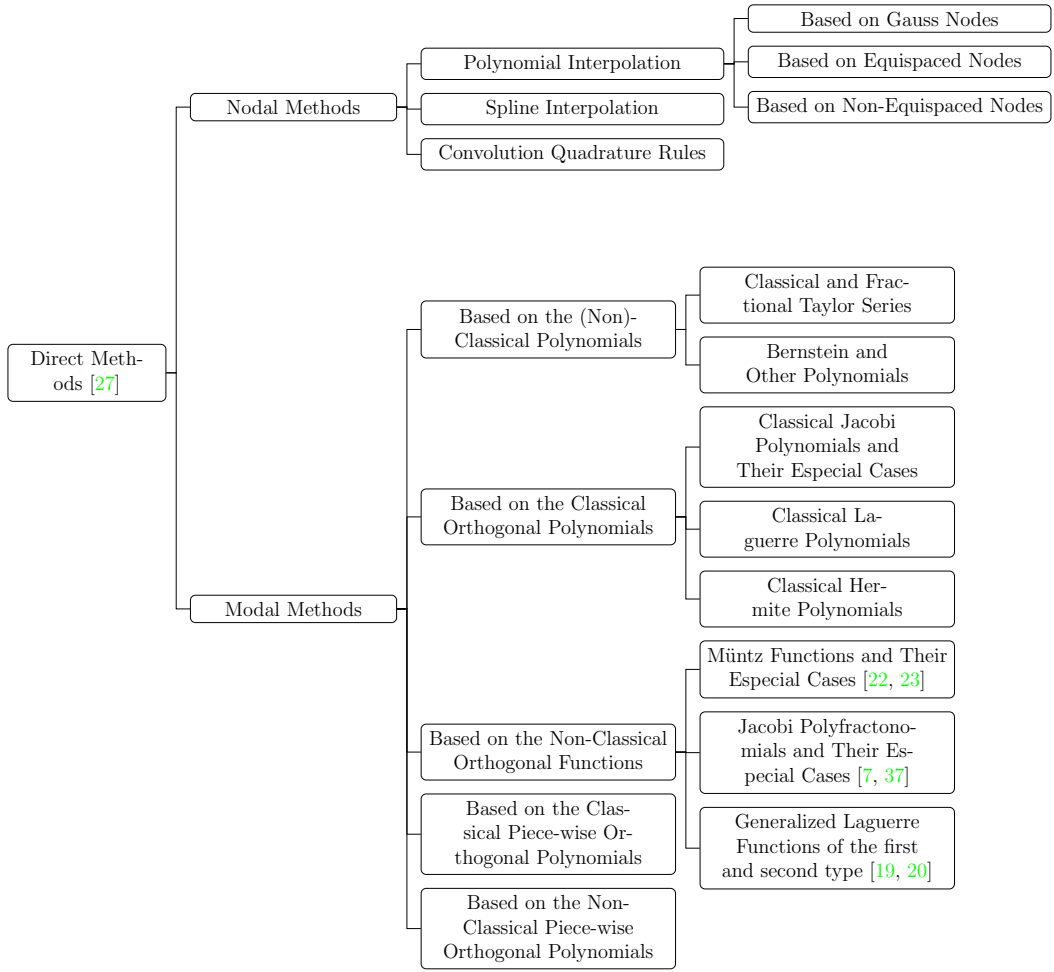
\captionof{figure}{Diagram of the Direct Methods to approximate Caputo fractional derivative. }
	\label{DM}
\end{center}

The outline of this paper is as follows. In Section \ref{Sec_2} two new diffusive representations (DR)  for the Caputo fractional derivative, numerical parts of the new methods and error analysis of them are presented. In Section \ref{Sec_3} some numerical results together with an improvement of the new method are given. Finally, in Section \ref{Sec_4} concluding remarks and some future works are proposed.  
\section{The sine and cosine diffusive representations}\label{Sec_2}
In this section, we introduce two new diffusive representations (DRs) which we will call the sine and cosine diffusive representations (SDR and CDR) for Caputo  fractional derivative. To reach our aim, we commence with the preliminary definitions and theorems.
\begin{Definition}\cite{MR2218073,MR2680847} Let $\alpha>0$. The Caputo fractional derivative is defined as:
	\begin{equation}\label{Cap}
		{}^{C}D_{0^+}^{\alpha}y(t)=\frac{1}{\Gamma(\lceil \alpha \rceil-\alpha)}\int_{0}^{t}(t-\tau)^{\lceil \alpha \rceil-\alpha-1}y^{(\lceil \alpha \rceil)}(\tau)\,d\tau,
	\end{equation}
	where $\lceil .\rceil$ stands for the ceiling function, that rounds up to the next integer not less than its argument. 
	
	For $0<\alpha<1$ we have:
	\begin{equation}\label{Cap1}
		{}^{C}D_{0^+}^{\alpha}y(t)=\frac{1}{\Gamma(1-\alpha)}\int_{0}^{t}(t-\tau)^{-\alpha}y'(\tau)\,d\tau.
	\end{equation}
\end{Definition} 
For more information and properties of the Caputo fractional derivative see  \cite{MR2218073,MR2680847}.  
Before going to introduce two new diffusive representations of Caputo fractional derivative of order $0<\alpha<1$, we recall  Yuan and Agrawal approach \cite{Yuan2002}.
\begin{Theorem}
	(Yuan and Agrawal approach (YA)) Let   $0<\alpha<1$. We have the following diffusive representation of the Caputo fractional derivative:
	\begin{equation*}\label{Eq_0}
		{}^{C}D_{0^+}^{\alpha}y(t)=\int_{0}^{\infty}z^{2\alpha-1} \omega^{YA}(z,t)\,dz,
	\end{equation*}
	where
	\[
	\omega^{YA}(z,t)=\frac{2\sin(\pi\alpha)}{\pi}\left(\int_{0}^{t}e^{-(t-\tau)z^2}y'(\tau)\,d\tau\right).
	\]
	It is easy to verify that $\omega^{YA}(z,t)$ satisfies the following differential equation:
	\begin{equation}\label{ODE1}
		\frac{\partial \omega^{YA}}{\partial t}+z^2\omega^{YA}=\frac{2\sin({\pi\alpha})}{\pi}y'(t),\ \omega^{YA}(z,0)=0.
	\end{equation}
\end{Theorem}
\begin{proof}
	See \cite{Yuan2002,Diethelm2008} for the proof of the theorem.
\end{proof}
\begin{Lemma}\cite{91138334} Let $0<\alpha<1$ and $b\in\Bbb{R}^+$. Then we have:
	\[
	\Gamma(\alpha)=\frac{b^\alpha}{\cos\left(\frac{\pi }{2}\alpha\right)}\int_{0}^{+\infty}t^{\alpha-1}\cos(bt)\,dt,
	\]
	and 
	\[
	\Gamma(\alpha)=\frac{b^\alpha}{\sin\left(\frac{\pi }{2}\alpha\right)}\int_{0}^{+\infty}t^{\alpha-1}\sin(bt)\,dt.
	\]	
\end{Lemma}
\begin{proof}
See \cite{91138334} for the proof of this lemma.
\end{proof}
Now and in this position, we will going to define two new DRs to approximate the Caputo fractional derivative.
\begin{Theorem}\label{Cos}
	(The cosine diffusive representation (CDR)). For $0<\alpha<1$, one can see
	\begin{equation*}\label{CDR}
		{}^{C}D_{0^+}^{\alpha}y(t)=\frac{2\sin(\tfrac{\pi\alpha}{2})}{\pi}\int_{0}^{\infty}z^{\alpha-1}\left(\int_{0}^{t}\cos\left((t-\tau)z\right){y'(\tau)}\,d\tau\right)\,dz=\int_{0}^{\infty}z^{\alpha-1}\omega^C(z,t)\,dz,
	\end{equation*}
	where
	\begin{equation}\label{Kernel_1}
		\omega^C(z,t)=\frac{2\sin(\tfrac{\pi\alpha}{2})}{\pi}\left(\int_{0}^{t}\cos\left((t-\tau)z\right){y'(\tau)}\,d\tau\right).
	\end{equation}
	\begin{proof}
		We first note that:
		\[
		\Gamma(\alpha)\Gamma(1-\alpha)=\frac{\pi}{\sin(\pi \alpha)},\ 0<\alpha<1.
		\]
		Then we have:
		\begin{eqnarray}\label{Eq_5_3}
			{}^{C}D_{0^+}^{\alpha}y(t)&=&\frac{1}{\Gamma(1-\alpha)}\int_{0}^{t}(t-\tau)^{-\alpha}y'(\tau)\,d\tau\nonumber\\
			&=&\frac{\Gamma(\alpha)}{\Gamma(1-\alpha)\Gamma(\alpha)}\int_{0}^{t}(t-\tau)^{-\alpha}y'(\tau)\,d\tau\nonumber\\
			&=&\frac{\sin(\pi\alpha)\Gamma(\alpha)}{\pi}\int_{0}^{t}(t-\tau)^{-\alpha}y'(\tau)\,d\tau\nonumber\\
			&=&\frac{\sin(\pi\alpha)x^{\alpha}}{\pi\cos\left(\tfrac{\pi \alpha}{2}\right)}\int_{0}^{t}\left(\int_{0}^{\infty}\theta^{\alpha-1}\cos(x\theta)\,d\theta\right)(t-\tau)^{-\alpha}y'(\tau)\,d\tau\nonumber\\
			&\stackrel{\theta=(t-\tau)z}{=}&\frac{2x^{\alpha}\sin\left(\tfrac{\pi \alpha}{2}\right)}{\pi}\int_{0}^{\infty}z^{\alpha-1}\left(\int_{0}^{t}\cos\left(x(t-\tau)z\right){y'(\tau)}\,d\tau\right)\,dz\nonumber.
		\end{eqnarray}
		If we take $x=1$,  theorem is proved.
	\end{proof}
\end{Theorem}
\begin{Theorem}\label{Sin}
	(The sine diffusive representation (SDR)). For $0<\alpha<1$, we have
	\begin{equation*}\label{SDR}
		{}^{C}D_{0^+}^{\alpha}y(t)=\frac{2\cos(\tfrac{\pi\alpha}{2})}{\pi}\int_{0}^{\infty}z^{\alpha-1}\left(\int_{0}^{t}\sin\left((t-\tau)z\right){y'(\tau)}\,d\tau\right)\,dz=\int_{0}^{\infty}z^{\alpha}\omega^S(z,t)\,dz,
	\end{equation*}
	where
	\begin{equation}\label{Kernel_2}
		\omega^S(z,t)=\frac{2\cos(\tfrac{\pi\alpha}{2})}{z\pi}  \left(\int_{0}^{t}\sin\left((t-\tau)z\right){y'(\tau)}\,d\tau\right).
	\end{equation}
	\begin{proof}
		The proof is similar to the proof of Theorem \ref{Cos}. 
	\end{proof}
\end{Theorem}
An important property of CDR and SDR is given in the next theorems.
\begin{Theorem}\label{Diff}
	Let $0<\alpha<1$.
	\begin{itemize}
		\item  For a given function $y$ for which its second derivative exists on $[0,T]$, $\omega^C(z,t)$ (for fixed $z>0$) satisfies the following second-order differential equation:
		\begin{equation}\label{ODE2}
			\begin{cases}
				\displaystyle\frac{\partial^2 \omega^C}{\partial t^2}+z^2\omega^C=\frac{2\sin(\tfrac{\pi\alpha}{2})}{\pi} y''(t),\\
				\displaystyle\omega^C(z,0)=0,\ \frac{\partial}{\partial t}\omega^C(z,0)=\frac{2\sin(\tfrac{\pi\alpha}{2})}{\pi} y'(0).
			\end{cases}
		\end{equation}
		\item For a given function $y$ for which its first derivative exists on $[0,T]$, $\omega^S(z,t)$ (for fixed $z>0$) satisfies the following second-order differential equation:
		\begin{equation}\label{ODE3}
			\begin{cases}
				\displaystyle\frac{\partial^2 \omega^S}{\partial t^2}+z^2\omega^S=\frac{2\cos(\tfrac{\pi\alpha}{2})}{\pi}\ y'(t),\\ 	\displaystyle \omega^S(z,0)=\frac{\partial}{\partial t}\omega^S(z,0)=0.
			\end{cases}
		\end{equation}
	\end{itemize}
	\begin{proof}
		The proofs are straightforward.
	\end{proof}
\end{Theorem}

In the following some important remarks concerning  the mentioned second-order differential equations \eqref{ODE2} and \eqref{ODE3} are given.
\begin{Remark} The following remarks should be noted here: 
	\begin{itemize} 
		\item In contrast to the other types of the DRs which coupled with a first-order differential equation, the SDR and CDR are coupled with a second-order differential equations.
		\item As we know, the second-order differential equations \eqref{ODE2} and \eqref{ODE3} could be converted to a system of first-order differential equations. Thus, the SDR and CDR can be also considered as the classical DRs. 
		\item  The second derivative of the given function $y(t)$, which appears in Eq. \eqref{ODE2} seems as a bad point of CDR, but, in fact, as we will see later, in application we don't need to evaluate $y''(t)$.
	\end{itemize}
\end{Remark}
Due to the fact that the classical DRs of Caputo fractional derivative are usually coupled with a first-order differential equation, it is worthy to convert  \eqref{ODE2} and \eqref{ODE3} to the system of first-order differential equations.
\begin{Theorem}\label{SystemDiff}
	Let $0<\alpha<1$.
	\begin{itemize}
		\item  For a given function $y$ for which its second derivative exists on $[0,T]$, and  $\omega^C(z,t)=x_1(z,t)$ (for fixed $z>0$), then $x_1(z,t)$ satisfies in the following system of first-order differential equations:
		\begin{equation}\label{SysODE2}
			\begin{cases}
				\displaystyle\frac{\partial x_1}{\partial t}=x_2(z,t),\\
				\displaystyle\frac{\partial x_2}{\partial t}=-z^2x_1(z,t)+\frac{2\sin(\tfrac{\pi\alpha}{2})}{\pi} y''(t),\\
				\displaystyle x_1(z,0)=0,\ x_2(z,0)=\frac{2\sin(\tfrac{\pi\alpha}{2})}{\pi} y'(0).
			\end{cases}
		\end{equation}
		\item For a given function $y$ for which its first derivative exists on $[0,T]$, and assume $\omega^S(z,t)=x_1(z,t)$ (for fixed $z>0$), where $x_1(z,t)$ satisfies in the following system of first-order differential equations:
		\begin{equation}\label{SysODE3}
			\begin{cases}
				\displaystyle\frac{\partial x_1}{\partial t}=x_2(z,t),\\
				\displaystyle\frac{\partial x_2}{\partial t}=-z^2x_1(z,t)+\frac{2\cos(\tfrac{\pi\alpha}{2})}{\pi}\ y'(t),\\ 	\displaystyle x_1(z,0)=x_2(z,0)=0.
			\end{cases}
		\end{equation}
	\end{itemize}
	\begin{proof}
		The proofs are straightforward.
	\end{proof}
\end{Theorem}

The numerical parts of the newly presented methods CDR and SDR are given in the next subsection. 
\subsection{The numerical parts of CDR and SDR}\label{Sec_2.1}
An important question which arises here is that: How can one approximate the Caputo fractional derivative of order $0<\alpha<1$ based on CDR and SDR? 

In fact, the procedure to construct the desired approximation for the  Caputo fractional derivative of order $0<\alpha<1$ of the given function $y(t)$ has two steps:
\begin{enumerate}
	\item  In the first step, we need to choose a suitable quadrature formula to approximate the following semi-infinite integrals numerically:
	\begin{equation}\label{IntCos_1}
		\int_{0}^{+\infty}z^{\alpha-1}\omega^C(z,t)\,dz\approx\sum_{k=1}^{N}c_k\ \omega^C(z_k^c,t),
	\end{equation}
	and 
	\begin{equation}\label{IntSin_1}
		\int_{0}^{+\infty}z^{\alpha}\omega^S(z,t)\,dz\approx\sum_{k=1}^{N}s_k\ \omega^S(z_k^s,t),
	\end{equation}
	where $z_k^c$, $z_k^s$  and  $c_k$, $s_k$ are the nods and weights of the quadrature formulae, respectively. 
	\item In the second step, we need to approximate $\omega^C(z_k^s,t)$ and $\omega^S(z_k^s,t)$ numerically. To do so, we solve the obtained systems   of first-order \eqref{SysODE2} and \eqref{SysODE3} for $z=z_k^c$ and $z=z_k^s$, respectively. In practice, various numerical algorithms with step size $h$ like as Runge-Kutta scheme or a linear multistep method
	can be used to solve these systems of differential equations. We denote $\omega^C_h(z_k^s,t_i)$ and $\omega^S_h(z_k^s,t_i)$ as the numerical approximations of $\omega^C(z_k^s,t)$ and $\omega^S(z_k^s,t)$, respectively, obtained from systems  \eqref{SysODE2} and \eqref{SysODE3}, where $t\in[0,T]$ and $h$ is the step size with 
	\[
	h=\frac{T}{n-1},\ n\in\Bbb{N},
	\]
	and denote by $t_k=(k-1)h$, so $t_1=0$ and $ t_n=T$. 
	
	Now, substituting $\omega^C_h(z_k^s,t_i)$ and $\omega^S_h(z_k^s,t_i)$ into \eqref{IntCos_1} and \eqref{IntSin_1}, respectively, the Caputo fractional derivatives of order $0<\alpha<1$ of the given function $y(t)$ at $t_i,\ i=1,2,\cdots,n$ are obtained as:
	\begin{equation}\label{IntCos_1Frac}
		{}^{C}D_{0^+}^{\alpha}y(t)\Big|_{t=t_i}=\int_{0}^{+\infty}z^{\alpha-1}\omega^C(z,t_i)\,dz\approx\sum_{k=1}^{N}c_k\ \omega^C(z_k^c,t_i)\approx\sum_{k=1}^{N}c_k\ \omega^C_h(z_k^c,t_i),\ i=1,2,\cdots,n,
	\end{equation}
	and 
	\begin{equation}\label{IntSin_1Frac}
		{}^{C}D_{0^+}^{\alpha}y(t)\Big|_{t=t_i}=\int_{0}^{+\infty}z^{\alpha}\omega^S(z,t_i)\,dz\approx\sum_{k=1}^{N}s_k\ \omega^C(z_k^s,t_i)\approx\sum_{k=1}^{N}s_k\ \omega^S_h(z_k^s,t_i),\ i=1,2,\cdots,n.
	\end{equation}
\end{enumerate}

Next remark provides some important issues to obtain numerical approximations of Caputo fractional derivative of order $0<\alpha<1$.
\begin{Remark}
	For the first step of the numerical method which concerns with the use of a suitable quadrature rules, it is worthy to point out that various types of the quadrature rules have been proposed to handle the semi-infinite integral of the DR, recently. In the following we list some of them.
	
	To handle semi-infinite integral, in fact, Yuan and Agrawal proposed the classical Gauss-Laguerre quadrature rule \cite{Yuan2002}.  Then Lu and Hanyga suggested to split the semi-infinite integral into two integrals $[0,c]$ and $[c,+\infty)$ \cite{LuHanyga2005}. They used the Gauss-Jacobi and shifted Gauss-Laguerre rules to compute the integral over $[0,c]$  and  $[c,+\infty)$, respectively. Generalized Gauss-Laguerre and Gauss-Jacobi quadrature rules have been successfully carried out by K. Diethelm in \cite{Diethelm2008,Diethelm2009}. Composite  Gauss-Jacobi quadrature rule is used by Hinze et al. \cite{Hinze2019}. 
	
	For the second step, which provides some ODE solver, the backward Euler and trapezoidal methods have been suggested (See \cite{Diethelm2021,Diethelm2008} for more comments on the ODE solvers). 
\end{Remark}
We also point out that, in our computations, due to the asymptotic behaviors of the functions $\omega^C(z,t)$ and $\omega^S(z,t)$ when $z\to 0$ and $z\to+\infty$,  we will use the generalized Gauss-Laguerre formula to approximate the semi-infinite integrals (See Theorem \ref{AsymtoticCDRandSDR} and Remark \ref{RemWeight}). 

This means that for the CDR semi-infinite  integral the generalized Gauss-Laguerre with respect to the weight function $w(z)=z^{\alpha-1}e^{-z}$ is carried out. On the other hand, for $k=1,2,\cdots,n$, we write:
\begin{equation}\label{CDRINT}
	{}^{C}D_{0^+}^{\alpha}y(t)\Big|_{t=t_k}=
	\int_{0}^{\infty}z^{\alpha-1} e^{-z}\left[ e^{z}\omega^C(z,t_k)\right]\,dz\approx\sum_{i=1}^{N}w_i^{(\alpha-1)}\ e^{z_i^{(\alpha-1)}}\omega^C(z_i^{(\alpha-1)},t_k),
\end{equation}
where  $z_i^{(\alpha-1)}$ and $w_i^{(\alpha-1)}$ are the Gauss-Laguerre nodes and weights associated with the weight function $w(z)=z^{\alpha-1}e^{-z}$.

Now, we need to approximate $\omega^C(z_i^{(\alpha-1)},t_k)$  (by $\omega_h^C(z_i^{(\alpha-1)},t_k)$) numerically. So, the following backward Euler and the trapezoidal methods for \eqref{SysODE2} with $z=z_i^{(\alpha-1)},\ i=1,2,\cdots,N$ are suggested as:
\begin{equation}\label{Syatem_11}
	\begin{cases}
		x_1^E(z,t_k)=x_1^E(z,t_{k-1})+hx_2^E(z,t_{k-1}),\\
		x_2^E(z,t_k)=\frac{1}{1+z^2h^2}\left[x_2^E(z,t_{k-1})-z^2hx_1^E(z,t_{k-1})+\dfrac{2\sin(\tfrac{\pi\alpha}{2})}{\pi} \left(y'(t_k)-y'(t_{k-1})\right)\right],
	\end{cases},\ k=2,3,...,n,
\end{equation}
and (for $\ k=2,3,...,n$)
\begin{equation}\label{Syatem_111}
	\begin{cases}
		x_1^T(z,t_k)=x_1^T(z,t_{k-1})+\frac{h}{2}\left[x_2^T(z,t_{k-1})+x_2^E(z,t_{k})\right],\\
		x_2^T(z,t_k)=\frac{1}{1+\frac{z^2h^2}{4}}\left[\left(1-\frac{z^2h^2}{4}\right)x_2^T(z,t_{k-1})-z^2hx_1^T(z,t_{k-1})+\dfrac{2\sin(\tfrac{\pi\alpha}{2})}{\pi} \left(y'(t_k)-y'(t_{k-1})\right)\right],
	\end{cases}
\end{equation}
respectively, where
\begin{equation}\label{InitCDR}
	x_1(z,t_1)=0,\ x_2(z,t_1)=\frac{2\sin(\tfrac{\pi\alpha}{2})}{\pi} y'(0).
\end{equation}
Substituting  the solutions $x_1^E(z_i^{(\alpha-1)},t_k)$ and $x_1^T(z_i^{(\alpha-1)},t_k)$ as approximations of $\omega_h^C(z_i^{(\alpha-1)},t_k)$ into \eqref{CDRINT}, the approximations of the Caputo fractional derivative are obtained.
\begin{Remark}
	Due to the initial conditions \eqref{InitCDR} and to solve systems \eqref{Syatem_11} and \eqref{Syatem_111} numerically, we need to approximate $y'(0)$. So, we can use the following approximation formula:
	 \[
	y'(0)\approx \frac{y(h)-y(0)}{h}.
	\] 
\end{Remark}
Similarly, for the SDR semi-infinite  integral the generalized Gauss-Laguerre with respect to the weight function $w(z)=z^{\alpha}e^{-z}$ is used. This means that, for $k=1,2,\cdots,n$, we have:
\begin{equation}\label{SDRINT}
	{}^{C}D_{0^+}^{\alpha}y(t)\Big|_{t=t_k}=\int_{0}^{\infty}z^{\alpha} e^{-z}\left[e^{z}\omega^S(z,t_k)\right]\,dz\approx\sum_{i=1}^{N}w_i^{(\alpha)}\ e^{z_i^{(\alpha)}}\omega^S(z_i^{(\alpha)},t_k),
\end{equation}
where  $z_i^{(\alpha)}$ and $w_i^{(\alpha)}$ are the Gauss-Laguerre nodes and weights associated with the weight function $w(z)=z^{\alpha}e^{-z}$. 

Finally and similar to the previous method, to approximate $\omega^S(z_i^{(\alpha)},t_k)$ (by $\omega^S_h(z_i^{(\alpha)},t_k)$) the following
backward Euler method together with the trapezoidal method for \eqref{SysODE3} with $z=z_i^{(\alpha)},\ i=1,2,\cdots,N$ are given as:
\begin{equation}\label{Syatem_22}
	\begin{cases}
		x_1^E(z,t_k)=x_1^E(z,t_{k-1})+hx_2^E(z,t_{k-1}),\\
		x_2^E(z,t_k)=\frac{1}{1+z^2h^2}\left[x_2^E(z,t_{k-1})-z^2hx_1^E(z,t_{k-1})+\dfrac{2\cos(\tfrac{\pi\alpha}{2})}{\pi} \left(y(t_k)-y(t_{k-1})\right)\right],
	\end{cases},\ k=2,3,...,n,
\end{equation}
and (for $\ k=2,3,...,n$)
\begin{equation}\label{Syatem_222}
	\begin{cases}
		x_1^T(z,t_k)=x_1^T(z,t_{k-1})+\frac{h}{2}\left[x_2^T(z,t_{k-1})+x_2^E(z,t_{k})\right],\\
		x_2^T(z,t_k)=\frac{1}{1+\frac{z^2h^2}{4}}\left[\left(1-\frac{z^2h^2}{4}\right)x_2^T(z,t_{k-1})-z^2hx_1^T(z,t_{k-1})+\dfrac{2\sin(\tfrac{\pi\alpha}{2})}{\pi} \left(y(t_k)-y(t_{k-1})\right)\right],
	\end{cases}
\end{equation}
respectively, where
\[
x_1(z,t_1)=0,\ x_2(z,t_1)=0,
\]
Now, plugging the obtained solutions $x_1^E(z_i^{(\alpha)},t_k)$ and $x_1^T(z_i^{(\alpha)},t_k)$ as approximations of $\omega^S_h(z_i^{(\alpha)},t_k)$ into \eqref{SDRINT}, the approximations of Caputo fractional derivative can be obtained.
\subsection{Error analysis of CDR and SDR}\label{Sec_2.2}
The main goal of this subsection is to provide the error analysis of the approximation methods CDR and SDR when the parameters $N$ (the number of integration points in the generalized Gauss-Laguerre formula) and $h$ (the step size of the ODE solvers) vary. Thanks to the fact that the CDR and SDR approximations are constructed from two steps (ODE solver together with the quadrature formula), so, it is natural to take
into account both the errors of the generalized Gauss-Laguerre quadrature and  the ODE solvers, to obtain the complete error analysis of the introduced methods. To do so, we denote the CDR and SDR approximations of the Caputo fractional derivatives of order $\alpha\in(0,1)$ by:
\[
{}^{C}D_{0^+,C,N,h}^{\alpha}y(t)=\sum_{i=1}^{N}w_i^{(\alpha-1)}\ e^{z_i^{(\alpha-1)}}\omega^C_h(z_i^{(\alpha-1)},t),
\]
and 
\[
{}^{C}D_{0^+,S,N,h}^{\alpha}y(t)=\sum_{i=1}^{N}w_i^{(\alpha)}\ e^{z_i^{(\alpha)}}\omega^S_h(z_i^{(\alpha)},t),
\]
respectively. 

Now in the next subsections, the error analysis of  the generalized Gauss-Laguerre formula together with the ODE solver is provided.
\subsubsection{The contribution of the generalized Gauss-Laguerre formula }
To obtain the error analysis of the generalized Gauss-Laguerre formula,  we need to define:
\begin{eqnarray}\label{Resid_CDR}
	R^{\alpha}_{C,N,h}y(t)&:=&{}^{C}D_{0^+}^{\alpha}y(t)-{}^{C}D_{0^+,C,N,h}^{\alpha}y(t)\nonumber\\
	&=&\int_{0}^{+\infty}z^{\alpha-1}\omega^C(z,t)\,dz-\sum_{i=1}^{N}w_i^{(\alpha-1)}\ e^{z_i^{(\alpha-1)}}\omega^C_h(z_i^{(\alpha-1)},t)\nonumber\\
	&=&R_{C,N,\alpha-1}^{GGL}\left[ e^{(.)}\omega^C(.,t)\right]+Q_{C,N,\alpha-1}^{GGL}\left[ e^{(.)}E^C_h(.,t)\right], 
\end{eqnarray}
where 
\begin{equation}\label{ErrQuadCDR}
	R_{C,N,\alpha-1}^{GGL}\left[ e^{(.)}\omega^C(.,t)\right]:=\int_{0}^{+\infty}z^{\alpha-1}\omega^C(z,t)\,dz-\sum_{i=1}^{N}w_i^{(\alpha-1)}\ e^{z_i^{(\alpha-1)}}\omega^C(z_i^{(\alpha-1)},t),
\end{equation}
is the error of generalized Gauss-Laguerre and
\begin{equation}\label{ErrODECDR}
	Q_{C,N,\alpha-1}^{GGL}\left[ e^{(.)}E^C_h(.,t)\right]:=\sum_{i=1}^{N}w_i^{(\alpha-1)}\ e^{z_i^{(\alpha-1)}}\left[\omega^C(z_i^{(\alpha-1)},t)-\omega^C_h(z_i^{(\alpha-1)},t)\right],
\end{equation}
is the error of the ODE solver. Similarly, we denote:
\begin{eqnarray}\label{Resid_SDR}
	R^{\alpha}_{S,N,h}y(t)&:=&{}^{C}D_{0^+}^{\alpha}y(t)-{}^{C}D_{0^+,S,N,h}^{\alpha}y(t)\nonumber\\
	&=&\int_{0}^{+\infty}z^{\alpha}\omega^S(z,t)\,dz-\sum_{i=1}^{N}w_i^{(\alpha)}\ e^{z_i^{(\alpha)}}\omega^S_h(z_i^{(\alpha)},t)\nonumber\\
	&=&R_{S,N,\alpha}^{GGL}\left[ e^{(.)}\omega^S(.,t)\right]+Q_{S,N,\alpha}^{GGL}\left[ e^{(.)}E^S_h(.,t)\right], 
\end{eqnarray}
where 
\begin{equation}\label{ErrQuadSDR}
	R_{S,N,\alpha}^{GGL}\left[ e^{(.)}\omega^S(.,t)\right]:=\int_{0}^{+\infty}z^{\alpha}\omega^S(z,t)\,dz-\sum_{i=1}^{N}w_i^{(\alpha)}\ e^{z_i^{(\alpha)}}\omega^S(z_i^{(\alpha)},t),
\end{equation}
denotes the error of generalized Gauss-Laguerre and
\begin{equation}\label{ErrODESDR}
	Q_{S,N,\alpha}^{GGL}\left[ e^{(.)}E^S_h(.,t)\right]:=\sum_{i=1}^{N}w_i^{(\alpha)}\ e^{z_i^{(\alpha)}}\left[\omega^S(z_i^{(\alpha)},t)-\omega^S_h(z_i^{(\alpha)},t)\right],
\end{equation}
used for the error of the ODE solver.

In this position, we start to analysis the errors of the generalized Gauss-Laguerre formulae \eqref{ErrQuadCDR} and \eqref{ErrQuadSDR}. To reach this aim, we need to have the behaviour of the integrands $\omega^C(z,t)$ and $\omega^S(z,t)$ when $z\to0$ and $z\to\infty$.
  Here, the symbol $a(v)\sim b(v)$ means that
there exist two strictly positive constants $A$ and $B$ such that: 
\[
\left|\frac{a(v)}{b(v)}\right|\in [A, B],
\] 
as $v$ tends to the indicated limit.
\begin{Theorem}\label{AsymtoticCDRandSDR}
	Let $t \in(0, T)$ be fixed and $0<\alpha<1$.
	\begin{enumerate}
		\item[(a)] Assume that there exists some constant $C>0$, such that for all $t\in(0,T)$ we have  $|y'(t)|>C$ then functions $\omega^C(.,t)$ and  $\omega^S(.,t)$ defined  in \eqref{Kernel_1} and \eqref{Kernel_2}, respectively behave as:
		\begin{eqnarray}\label{Asymp_1}
			&&z^{\alpha-1}\omega^C(z,t)\sim z^{\alpha-1}\ \ \ \text{as}\ z\to 0,\\
			&& z^{\alpha}\omega^S(z,t)\sim z^{\alpha}\ \ \ \text{as}\ z\to 0.
		\end{eqnarray}
		\item[(b)] Let $y(t)\in C^2[0,T]$ and $y(0)=y'(0)=0$. Assume that $y(t)$ and $y'(t)$ are  of exponential order, then we have:
		\begin{eqnarray}\label{Asymp_2}
			&&z^{\alpha-1}\omega^C(z,t)\sim z^{\alpha-3}\ \ \ \text{as}\ z\to +\infty.
		\end{eqnarray}
		\item[(c)] Let $y(t)\in C^1[0,T]$ and $y(t)$ and $y'(t)$ be  of exponential order then we have:
		\begin{eqnarray}\label{Asymp_3}
			&& z^{\alpha}\omega^S(z,t)\sim z^{\alpha-2}\ \ \ \text{as}\ z\to +\infty.
		\end{eqnarray}
	\end{enumerate}
\begin{proof}
	For  part (a), using the integration by part, yields:
	\begin{eqnarray*}
		\int_{0}^{t}\cos((t-\tau)z)y'(\tau)\,d\tau&=&y(\tau)\cos((t-\tau)z)\Bigg]_{\tau=0}^{\tau=t}-z\int_{0}^{t}\sin((t-\tau)z)y(\tau)\,d\tau\\
		\\
		&=&y(t)-\cos(tz)y(0)-z\int_{0}^{t}\sin((t-\tau)z)y(\tau)\,d\tau.
	\end{eqnarray*}
	For fixed $t$, the right side integral remains bounded as $z\to0$, that proves: 
	\[
	\lim_{z\to0}\int_{0}^{t}\cos((t-\tau)z)y'(\tau)\,d\tau=\lim_{z\to0}\left[y(t)-\cos(tz)y(0)\right]=y(t)-y(0).
	\]
	Substituting the above relation into \eqref{Kernel_1},  completes the proof. 
	
	For function $\omega^S(z,t)$, we write:
	\begin{eqnarray*}
		\frac{1}{z}\int_{0}^{t}\sin((t-\tau)z)y'(\tau)\,d\tau&=&\frac{1}{z}y(\tau)\sin((t-\tau)z)\Bigg]_{\tau=0}^{\tau=t}+\int_{0}^{t}\cos((t-\tau)z)y(\tau)\,d\tau\\
		\\
		&=&-\frac{\sin(tz)}{z}y(0)+\int_{0}^{t}\cos((t-\tau)z)y(\tau)\,d\tau.
	\end{eqnarray*}
	Now, we have:
	\begin{eqnarray*}
		\lim_{z\to0}\frac{1}{z}\int_{0}^{t}\sin((t-\tau)z)y'(\tau)\,d\tau&=&\lim_{z\to0}\left[-\frac{\sin(tz)}{z}y(0)+\int_{0}^{t}\cos((t-\tau)z)y(\tau)\,d\tau\right]\\
		&=&-ty(0)+\int_{0}^{t}y(\tau)\,d\tau.
	\end{eqnarray*}
	The above relation together with \eqref{Kernel_2}, concludes the proof.
	
	For  part (b), thanks to the fact that $y(t)$ and $y'(t)$ are continuous and  of exponential order and then using the Laplace transform, formally gives:
	\begin{eqnarray*}
		z^2\mathcal{L}\left\{\int_{0}^{t}\cos((t-\tau)z)y'(\tau)\,d\tau\right\}&=&z^2	\mathcal{L}\left\{\cos(tz)\right\}\mathcal{L}\left\{y'(t)\right\}
		=z^2\frac{s}{s^2+z^2}\left(s\mathcal{L}\left\{y(t)\right\}-y(0)\right).
	\end{eqnarray*}
	Thus,
	\[
	\displaystyle	z^2\int_{0}^{t}\cos((t-\tau)z)y'(\tau)\,d\tau=\mathcal{L}^{-1}\left\{z^2\frac{s}{s^2+z^2}\left(s\mathcal{L}\left\{y(t)\right\}-y(0)\right)\right\}.
	\]
	Now by taking the limit when $z\to+\infty$, we formally obtain:  
	\begin{eqnarray*}
		\displaystyle	\lim_{z\to+\infty}z^2\int_{0}^{t}\cos((t-\tau)z)y'(\tau)\,d\tau&=&\mathcal{L}^{-1}\left\{\lim_{z\to+\infty}\left[z^2\frac{s^2\mathcal{L}\left\{y(t)\right\}}{s^2+z^2}\right]\right\}\\
		&=&y''(t).
	\end{eqnarray*}
	Substituting the obtained result into \eqref{Kernel_1}, completes the proof. 
	
	Similarly, we can write:
	\begin{eqnarray*}
		\displaystyle	\lim_{z\to+\infty}z\int_{0}^{t}\sin((t-\tau)z)y'(\tau)\,d\tau&=&\mathcal{L}^{-1}\left\{\lim_{z\to+\infty}\left[z^2\frac{s\mathcal{L}\left\{y(t)\right\}-y(0)}{s^2+z^2}\right]\right\}=y'(t),
	\end{eqnarray*}
	Plugging the last relation into \eqref{Kernel_2}, the proof is concluded. 
\end{proof}
\end{Theorem}
\begin{Remark}\label{RemWeight}
	Let $0<\alpha<1$. Due to  Theorem \ref{AsymtoticCDRandSDR}, we have the following properties:
	\begin{itemize}
		\item The asymptotic behaviours of $z^{\alpha-1}\omega^C(z,t)$ and $z^{\alpha}\omega^S(z,t)$ when $z\to0$, indicate that the use of generalized Gauss-Laguerre with the weight functions $w(z)=z^{\alpha-1}e^{-z}$ and $w(z)=z^{\alpha}e^{-z}$, respectively, may lead to the smooth integrands at origin.
		\item  As we see, $z^{\alpha-1}\omega^C(z,t)$ and $z^{\alpha}\omega^S(z,t)$ when $z\to+\infty$ decay as $z^{\alpha-3}$ and $z^{\alpha-2}$, respectively. On the other hand, the exponent of $z$ for each case is always contained in $(-3,-2)$ and $(-2,-1)$, respectively. This fact is sufficient to make sure that the semi-infinite integrals \eqref{IntCos_1} and    \eqref{IntSin_1} exist.
	\end{itemize}
	
\end{Remark}
Now, in what follows, error analysis of the generalized Gauss-Laguerre formula is given.
\begin{Theorem}\label{ErrorQuad}
	Let  $y(t)$ and $y'(t)$ be of exponential order.
	\begin{itemize}
		\item(a): For $0<\alpha<1$ and $y\in C^2[0,T]$ such that $y(0)=y'(0)=0$, we have:
		\begin{equation}\label{CDEQuadEr}
			R_{C,N,\alpha-1}^{GGL}\left[ e^{(.)}\omega^C(.,t)\right]=\mathcal{O}(N^{\alpha-2}),
		\end{equation}
		for $t\in[0,T]$.
		\item(b): For $0<\alpha<1$ and $y\in C^1[0,T]$,  we also have:
		\begin{equation}\label{SDEQuadEr}
			R_{S,N,\alpha}^{GGL}\left[ e^{(.)}\omega^S(.,t)\right]=\mathcal{O}(N^{\alpha-1}),
		\end{equation}
		for $t\in[0,T]$.
	\end{itemize}
\end{Theorem}
\begin{proof}
	For the proof of this theorem see \cite{Diethelm2008}.
\end{proof}
\begin{Remark}
	Theorem \ref{ErrorQuad} states that for $0<\alpha<1$ the error of the generalized Gauss-Laguerre  quadrature rule of the CDR method when $N\to+\infty$ decays faster than the  SDR. 
\end{Remark}
\subsubsection{The contribution of the ODE solver}
The second part of the error analysis is about the truncation error of the ODE solver. To do this, we first note that, if $z_k^{(\gamma)},\ k=1,2,\cdots,N$ stands for the nodes of the generalized Laguerre integration formula with respect to the weight function $w(x)=x^{\gamma} e^{-x}$,  then we have $z_k^{(\gamma)}=4k+2\gamma+6$ \cite{Shen2011}.
To explain more clearly, we consider the following system of first order differential equations:
\begin{eqnarray}\label{SysGen}
	Y'(t)=F(t,Y(t)),\ Y(0)={\bf a},\ \ t\in[0,T],
\end{eqnarray}
for which
\begin{equation}
	Y(t)=\left[\begin{array}{c}
		y_1(t)\\
		y_2(t)\\
		\vdots\\
		y_m(t)
	\end{array}\right],\ \ 
	F(t,Y(t))=\left[\begin{array}{c}
		f_1(t,y_1,\cdots,y_m)\\
		f_2(t,y_1,\cdots,y_m)\\
		\vdots\\
		f_m(t,y_1,\cdots,y_m)
	\end{array}\right],\ \ 
	{\bf a}=\left[\begin{array}{c}
		y_1(0)\\
		y_2(0)\\
		\vdots\\
		y_m(0)
	\end{array}\right],
\end{equation}
where $F: [0,T]\times \Bbb{R}^m\longrightarrow \Bbb{R}^m$ is continuous in its first variable and satisfies a Lipschitz condition with constant $L$ in its second variable, i.e.,  for any $t\in[0,T]$ and $Z,W\in \Bbb{R}^m$, we have
\[
\|F(t, W)-F(t,Z)\|\leq L\|W-Z\|.
\]
The convergence of a numerical method applied to Eq. \eqref{SysGen} requires the step size $h$ to satisfy in $Lh<1$. In  our case, viz.  Eqs. \eqref{ODE2} and \eqref{ODE3}, we have $L=\left(z_k^{(\gamma)}\right)^2,\  k=1,2,\cdots, N$.  This may lead to some difficulties from the numerical point of view for sufficiently large $N$ unless the step sizes $h$ are chosen extremely small. For this reason, we always assume that
\[
h\left(z_N^{(\gamma)}\right)^2<1\Longrightarrow h<\frac{1}{\left(z_N^{(\gamma)}\right)^2}\sim N^{-2}.
\] 
So, we have the following lemma.
\begin{Lemma}\label{Conv_ODE}
	Assume that a A-stable one-step implicit method of order $p$ is used for  Eqs. \eqref{ODE2} and \eqref{ODE3}, then there exists a constant $C>0$ such that:
	\begin{eqnarray}
		&&\left|E^C_h(z_k^{(\alpha-1)},t)\right|\leq Ch^p e^{3T\displaystyle \left(z_k^{(\alpha-1)}\right)^2 },\ \\
		&& \left|E^S_h(z_k^{(\alpha)},t)\right|\leq Ch^p e^{3T\displaystyle \left(z_k^{(\alpha)}\right)^2 }, 
	\end{eqnarray}
	for $k=1,2,\cdots,N$, sufficiently small $h>0$ and any $t\in[0,T]$. 
	\begin{proof}
		For the proof of this lemma see \cite{Diethelm2008}.
	\end{proof}
\end{Lemma}
\begin{Theorem}\label{BoundODESolver}
	Under the assumptions of the previous Lemma, there exist constants $C_1 > 0,\ C_2>0$
	such that:
	\begin{eqnarray}
		&& \Big|Q_{C,N,\alpha-1}^{GGL}\left[ e^{(.)}E^C_h(.,t)\right]\Big|\leq C_1h^{p}\int_{0}^{4N}e^{3T\displaystyle z^2 }\,dz,\\
		&&\Big|Q_{S,N,\alpha}^{GGL}\left[ e^{(.)}E^S_h(.,t)\right]\Big|\leq C_2h^{p}\int_{0}^{4N}e^{3T\displaystyle z^2 }\,dz.
	\end{eqnarray}
\end{Theorem}
\begin{proof}
	The proof of this theorem is fairly similar to the proof of Theorem 5 of \cite{Diethelm2008}.
\end{proof}
\subsubsection{ The overall error analysis  }
For the reader's convenience, summary of the error analysis  is given in the following theorem.
\begin{Theorem}\label{OverErrCDRSDR}
	Let $0<\alpha<1$. If a A-stable one-step implicit method of order $p$ with the step size $h<N^{-2}$, (where $N$ is the number of integration points in the generalized Gauss-Laguerre formula) is used for  Eqs. \eqref{ODE2} and \eqref{ODE3}, then the overall error analysis of CDR and SDR approximation formulae satisfies:
	\begin{itemize}
		\item If $y(t)\in C^1[0,T]$, then for $t\in[0,T]$, we have
		\begin{equation}\label{OVERERRCDR}
			\Big|R^{\alpha}_{C,N,h}y(t)\Big|=\mathcal{O}(N^{\alpha-2})+\mathcal{O}(h^p)\int_{0}^{4N}e^{3T\displaystyle z^2 }\,dz.
		\end{equation}
		\item If $y(t)\in C^1[0,T]$, then for $t\in[0,T]$, we have
		\begin{equation}\label{OVERERRSDR}
			\Big|R^{\alpha}_{S,N,h}y(t)\Big|=\mathcal{O}(N^{\alpha-1})+\mathcal{O}(h^p)\int_{0}^{4N}e^{3T\displaystyle z^2 }\,dz.
		\end{equation}
	\end{itemize}
\end{Theorem}
\begin{proof}
	The proofs are immediately obtained from Theorems \ref{ErrorQuad} and \ref{BoundODESolver}.
\end{proof}
To have a good sense and in order to compare the CDR and SDR methods and the Yuan and Agrawal (YA) one, the error analysis of their method is provided here \cite{Diethelm2008}.
\begin{Theorem}\label{OverErrYA}
	Let $0<\alpha<1$ and $y(t)\in C^1[0,T]$. If a A-stable one-step implicit method of order $p$ with the step size $h<N^{-2}$ is used for  Eq. \eqref{ODE1}, then the overall error analysis of YA approximation formula satisfies:
	\begin{equation}\label{OVERERRYA}
		\Big|R^{\alpha}_{N,h}y(t)\Big|=\mathcal{O}(N^{2\alpha-2})+\mathcal{O}(h^p)\int_{0}^{4N}e^{3T\displaystyle z^2 }\,dz,
	\end{equation}
	for $t\in[0,T]$.
\end{Theorem}
\begin{proof}
	See Theorem 6 of \cite{Diethelm2008}.
\end{proof}
\section{Numerical results}\label{Sec_3}
In this position, we proceed to testify the numerical methods with some examples. To make a good comparison, the proposed methods CDR and SDR have compared with  the Yuan and Agrawal method (YA). To do so, we first denote 
\[
E_{\infty}(N)=\max_{t\in[a,b]}\left|{}^{C}D_{0^+}^{\alpha}y(t)-{}^{C}D_{0^+,N,h}^{\alpha}y(t)\right|,
\]
for the maximum errors obtained by the methods YA, CDR and SDR for fixed $N$  and $t$ varies on the domain $[a,b]$.
\begin{Example}\label{Ex_1}
	For the first example we consider \cite{Diethelm2021,Diethelm2008}
	\[
	y(t)=t^{1.6},\ \ t\in[0,3],
	\]
	where 
	\[
	{}^{C}D_{0^+}^{\alpha}y(t)=\frac{\Gamma(2.6)}{\Gamma(2.6-\alpha)}t^{1.6-\alpha},\ \alpha=0.4.
	\]
	We also note that $y\in C^1[0,3]$. Relative errors of the approximation of ${}^{C}D_{0^+}^{\alpha}y(t)$ by the backward Euler method for some values of $n$  with the use of  the generalized Gauss-Laguerre quadrature rule with $N=50$-points, obtained from YA, CDR and SDR methods are plotted in Figs. \ref{Fig_1}-\ref{Fig_2}. Also, a comparison of  maximum absolute errors of approximations obtained from YA, CDR and SDR methods for some values of $N$ and $n=10^4$  is shown in Fig. \ref{Fig_3}.
	\newline
	
	\begin{figure}[H]
		\vspace{-5cm}
		\includegraphics[width=12cm]{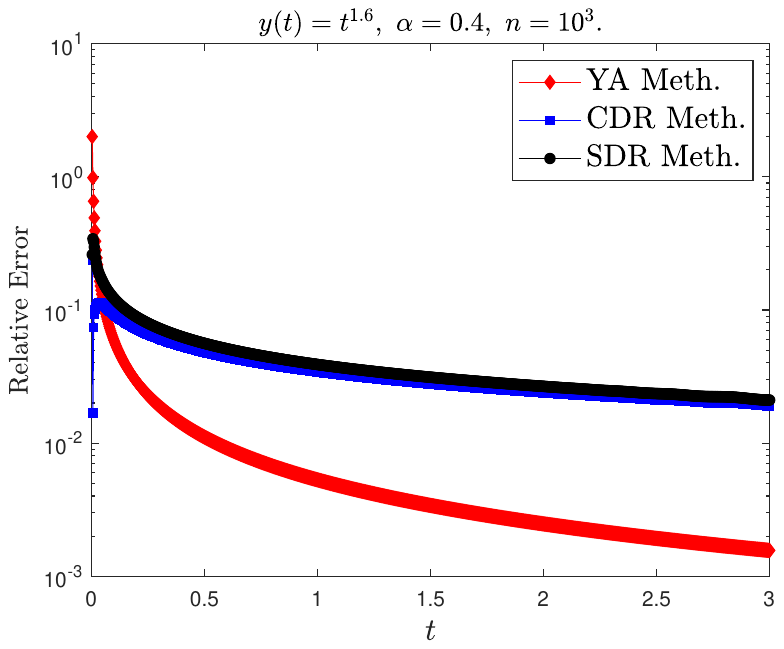}\hspace{-3cm}\includegraphics[width=12cm]{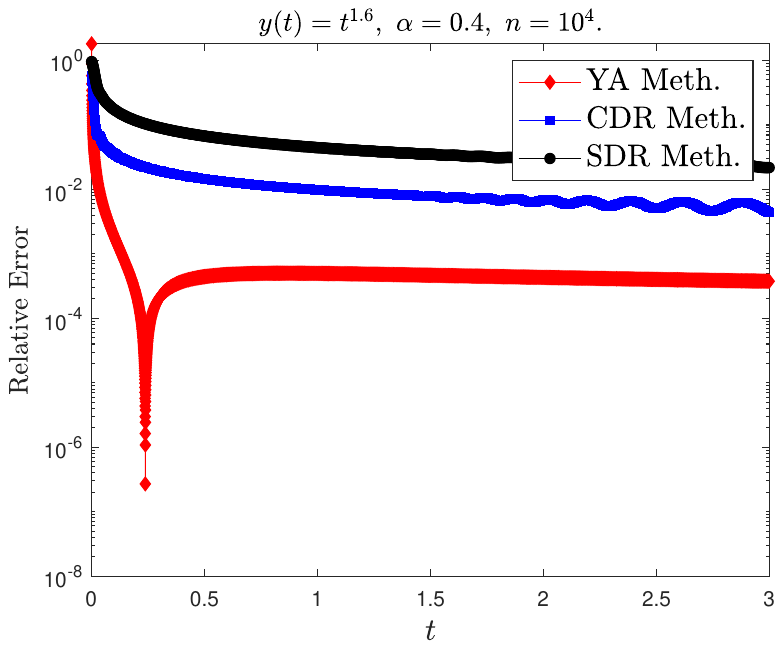}\\ \vspace{-5cm}
		\caption{Relative errors obtained by the backward Euler method for  three  methods YA, CDR and SDR for $n=10^3$ and $n=10^4$.}\label{Fig_1}
	\end{figure}
	\begin{figure}[H]
		\vspace{-4cm}
		\includegraphics[width=12cm]{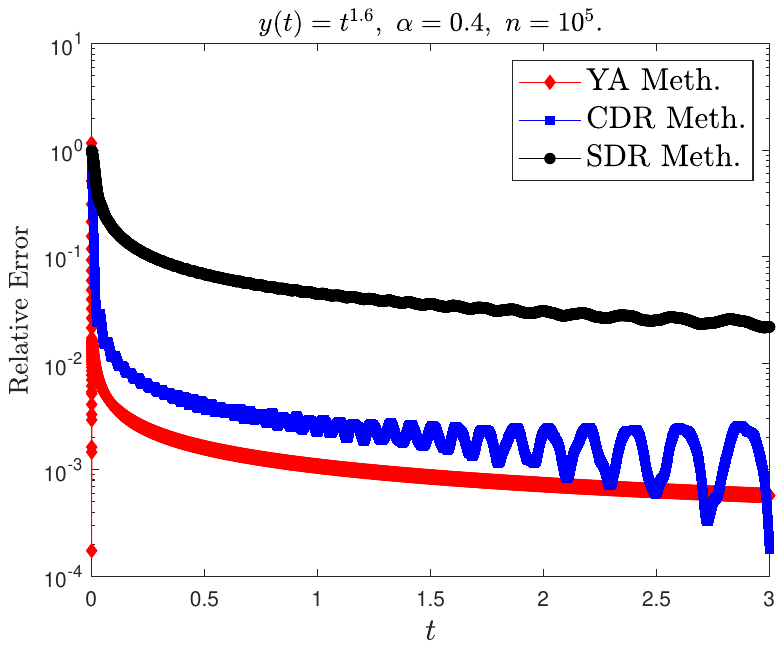}\hspace{-3cm}\includegraphics[width=12cm]{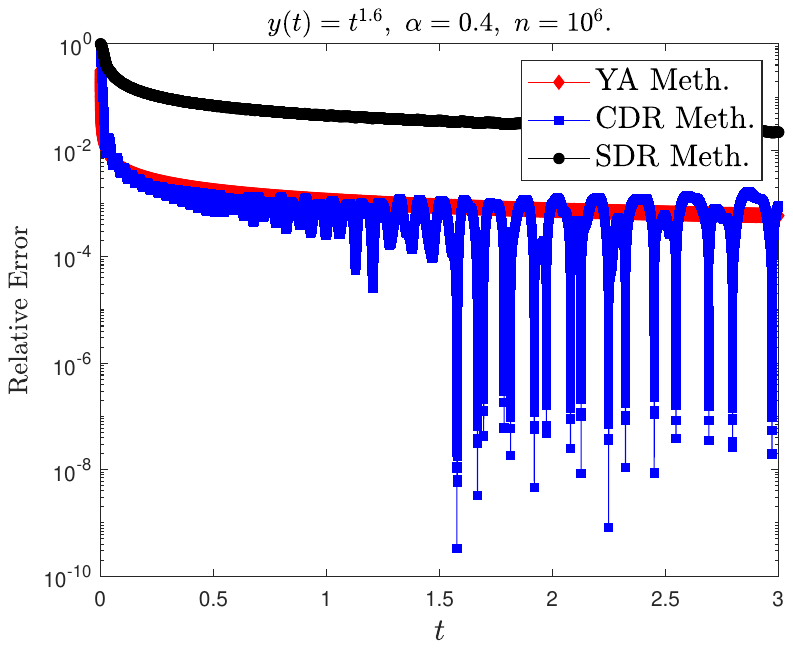}\\ \vspace{-5cm}
		\caption{Relative errors obtained by the backward Euler method for  three  methods YA, CDR and SDR  for $n=10^5$ and $n=10^6$.}\label{Fig_2}
	\end{figure}
	\begin{figure}[H]
		\vspace{-4cm}
		\includegraphics[width=7.cm,height=15cm]{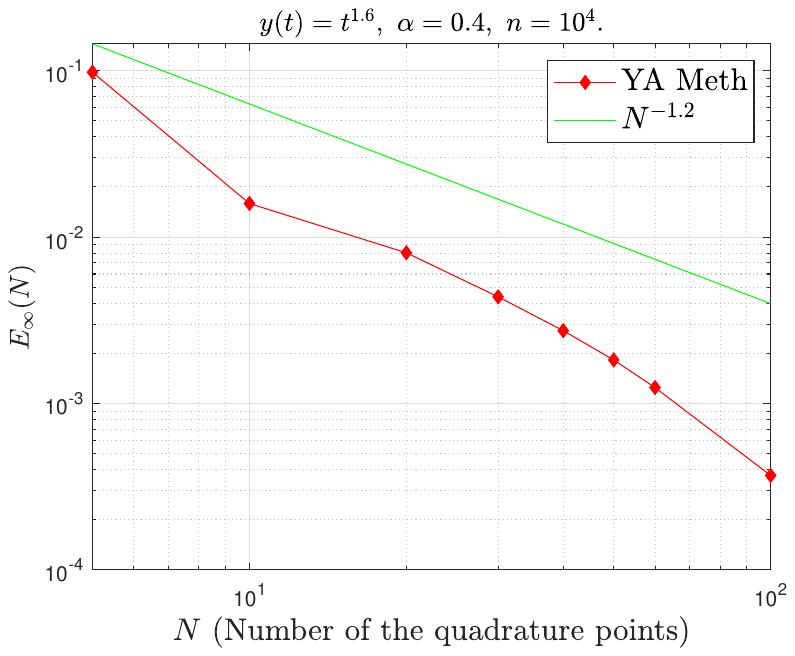}\hspace{-1.5cm}\includegraphics[width=7.cm,height=15cm]{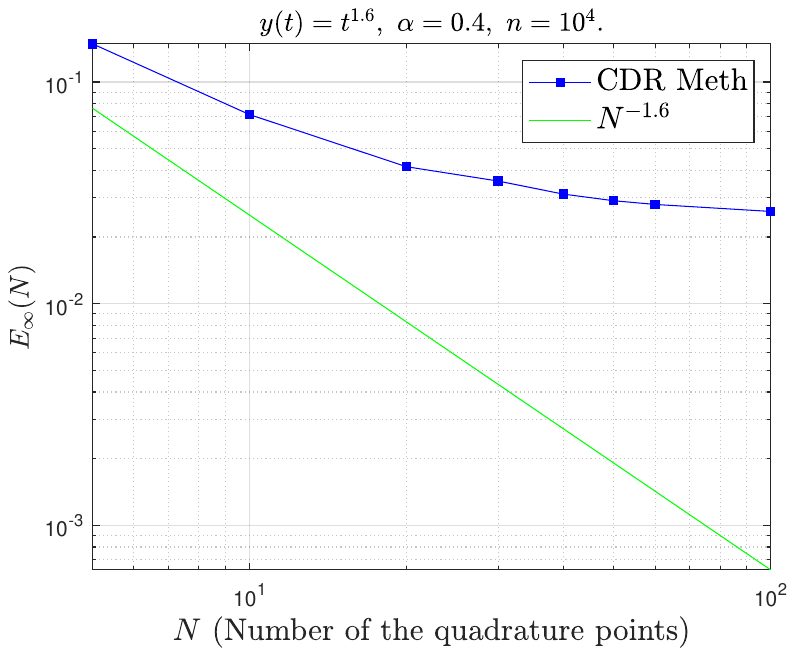}\hspace{-1.5cm}\includegraphics[width=7.cm,height=15cm]{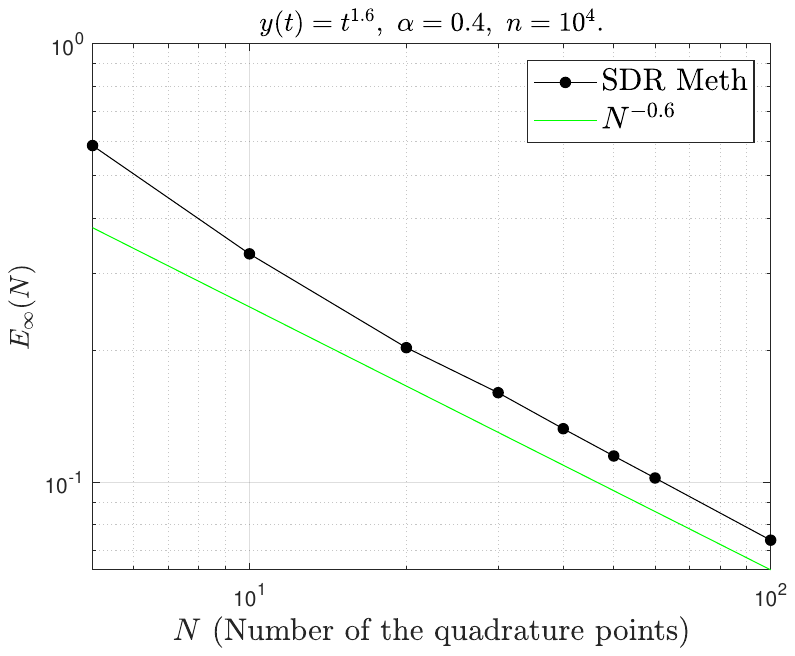}\\ \vspace{-2cm}
		\vspace{-2.5cm}
		\caption{Maximum errors obtained by the backward Euler method for three  methods YA, CDR and SDR with $n=10^4$ and various values of $N$.}\label{Fig_3}
	\end{figure}
\end{Example}
\begin{Example}\label{Ex_2}
	For the second example, consider the following sufficiently smooth function  \cite{Diethelm2021,Diethelm2008}:
	\[
	y(t)=t^{3},\ \ t\in[0,1],
	\]
	where 
	\[
	{}^{C}D_{0^+}^{\alpha}y(t)=\frac{\Gamma(4)}{\Gamma(4-\alpha)}t^{3-\alpha},\ \alpha=0.6.
	\]
	We also have, $y\in C^{\infty}[0,1]$.
	Similarly, the relative errors of approximation ${}^{C}D_{0^+}^{\alpha}y(t)$ (with $N=50$) for  three methods YA, CDR and SDR are also  plotted in Figs. \ref{Fig_4} and \ref{Fig_5}.
	A comparisons of the maximum errors of approximation ${}^{C}D_{0^+}^{\alpha}y(t)$ (with $N=50$) is shown in Fig. \ref{Fig_6}.
	
	\begin{figure}[H]
		\vspace{-5cm}
		\includegraphics[width=12cm]{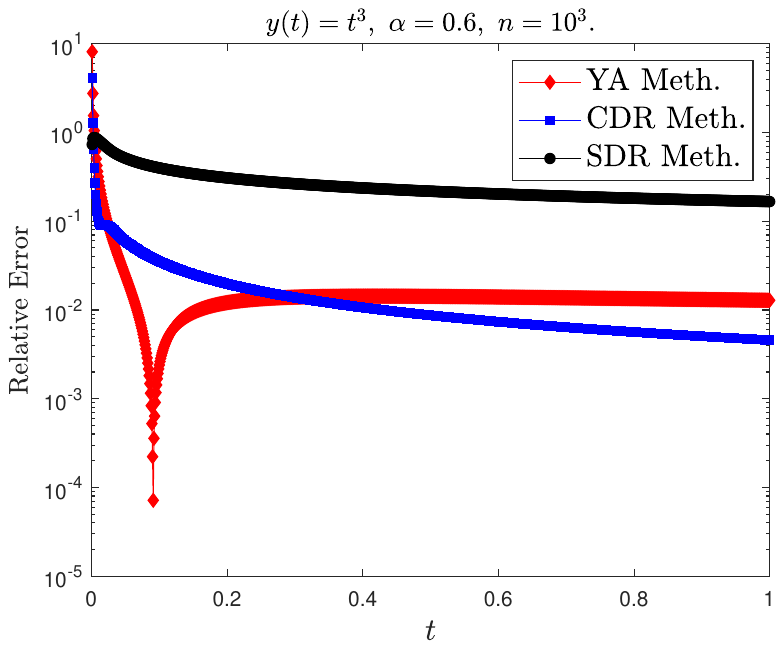}\hspace{-3cm}\includegraphics[width=12cm]{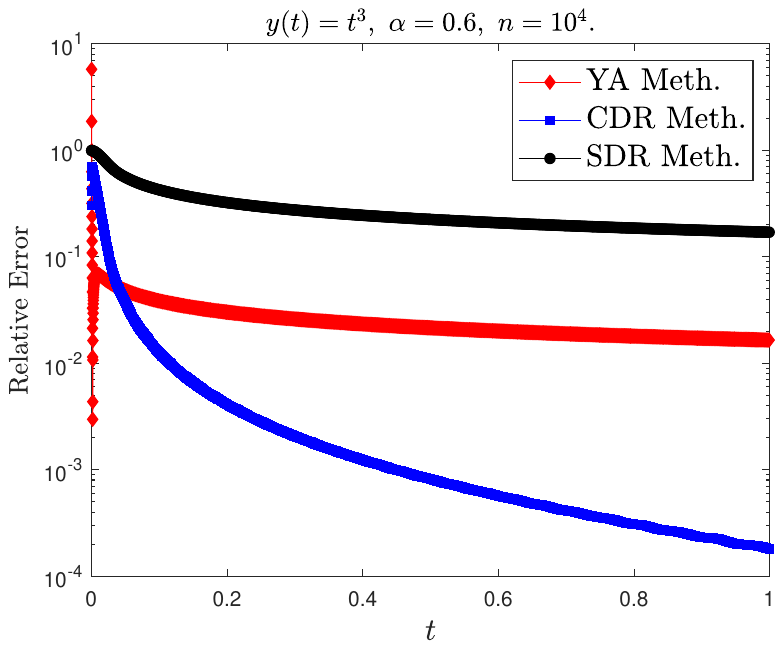}\\ \vspace{-5cm}
		\caption{The relative errors obtained by the backward Euler method for  three  methods YA, CDR and SDR for $n=10^3$ and $n=10^4$.}\label{Fig_4}
	\end{figure}
	\begin{figure}[H]
		\vspace{-4cm}
		\includegraphics[width=12cm]{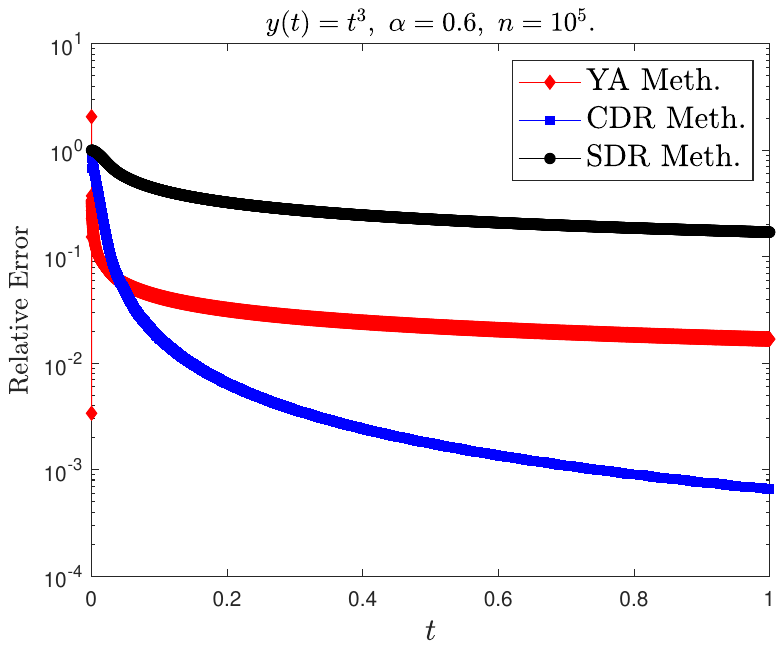}\hspace{-3cm}\includegraphics[width=12cm]{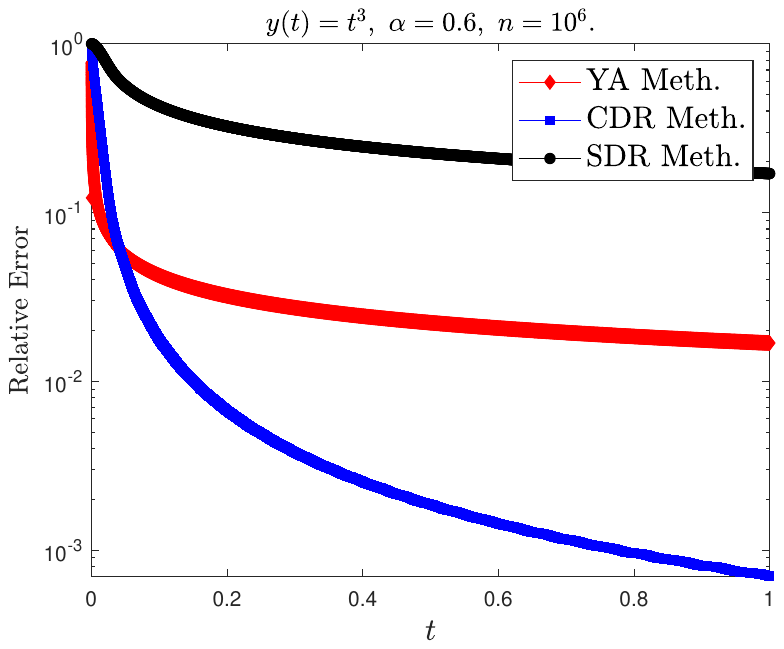}\\ \vspace{-5cm}
		\caption{The relative errors obtained by the backward Euler method for  three  methods YA, CDR and SDR for $n=10^5$ and $n=10^6$.}\label{Fig_5}
	\end{figure}
	\begin{figure}[H]
		\vspace{-4cm}
		\includegraphics[width=7cm,height=15cm]{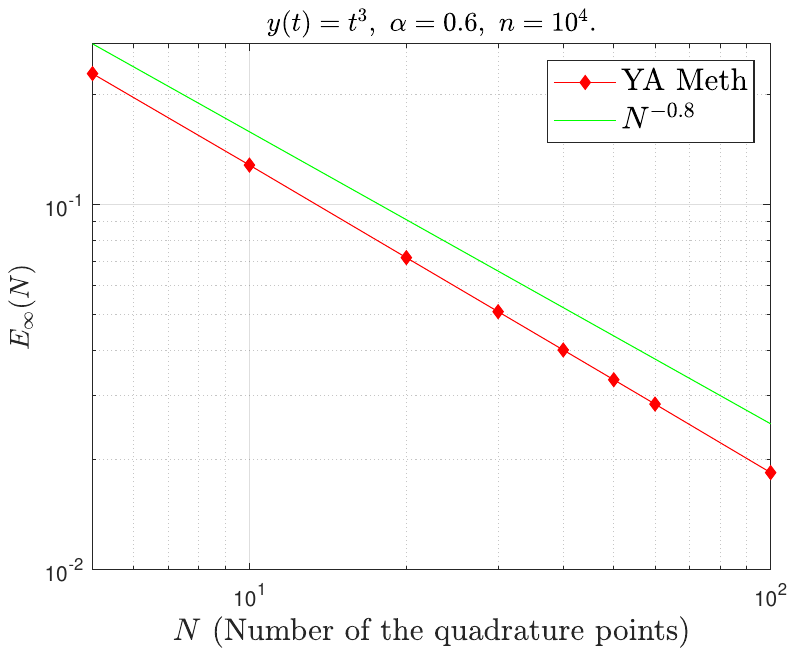}\hspace{-1.5cm}\includegraphics[width=7cm,height=15cm]{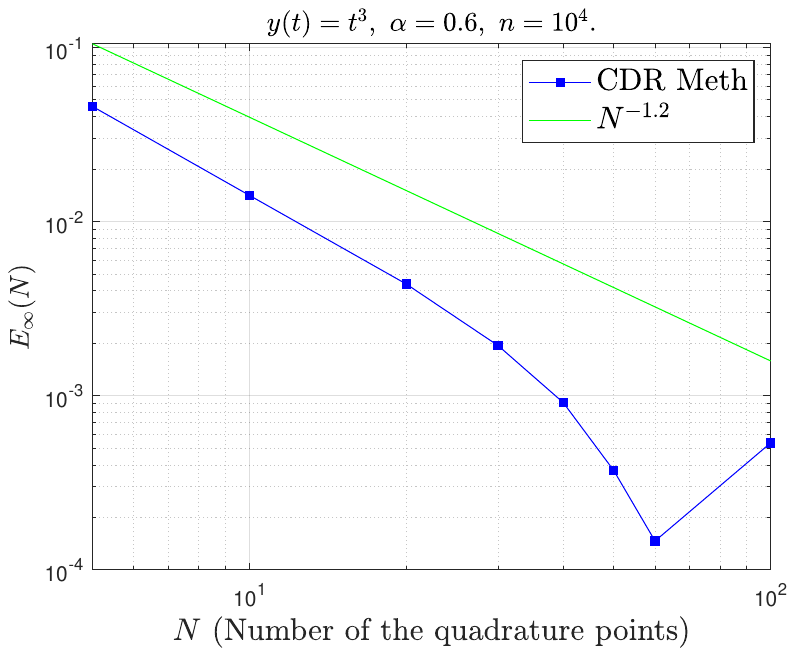}\hspace{-1.5cm}\includegraphics[width=7cm,height=15cm]{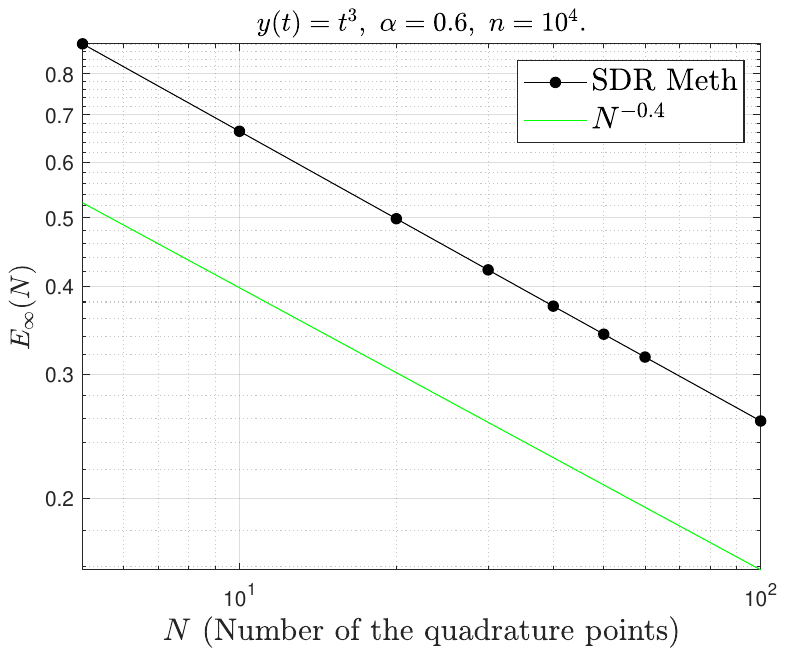}\\ \vspace{-2cm}
		\vspace{-2.5cm}
		\caption{Maximum norm of the errors obtained by the backward Euler method for three  methods YA, CDR and SDR with $n=10^4$ and various values of $N$.}\label{Fig_6}
	\end{figure}
\end{Example}
\begin{Example}\label{Ex_3}
	For the third example consider the sufficiently smooth and periodic function \cite{Sugiura2009}:
	\[
	y(t)=\sin t, \ t\in[0,1],
	\]	 
	where 
	\[
	{}^{C}D_{0^+}^{\alpha}y(t)=t^{1-\alpha}\sum_{k=0}^{+\infty}\frac{(-t)^{2k}}{\Gamma(2k+2-\alpha)},\ \alpha=0.5.
	\] 
	
	As we know,  $y\in C^{\infty}[0,1]$. The relative errors obtained by the YA, CDR and SDR methods to approximate ${}^{C}D_{0^+}^{\alpha}y(t)$ for $\alpha=0.5$ and $N=50$ versus some valued of $n$ have been reported in Figs. \ref{Fig_7} and \ref{Fig_8}. 
	
	We also report the maximum errors of the methods for some values of $N$ with $n=10^4$ in Fig. \ref{Fig_9}.
	\begin{figure}[H]
		\vspace{-5cm}
		\includegraphics[width=12cm]{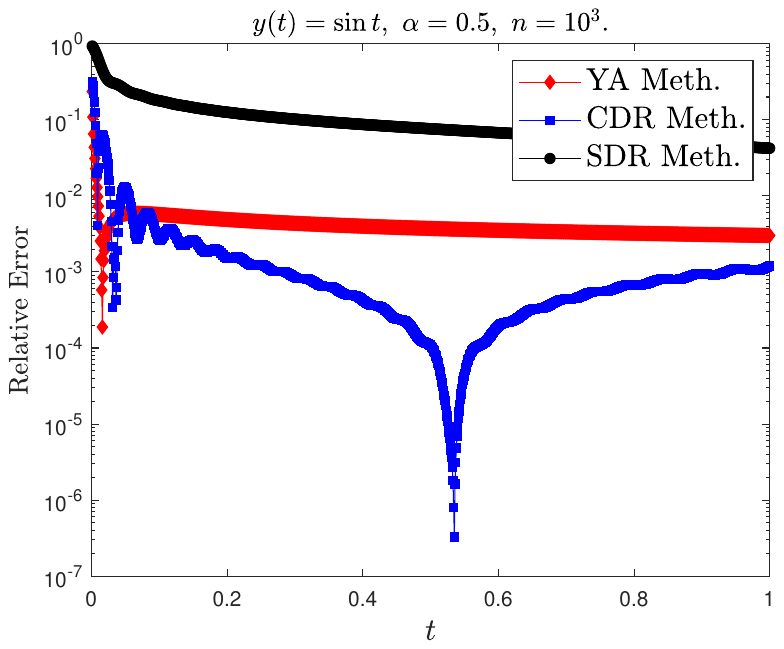}\hspace{-3cm}\includegraphics[width=12cm]{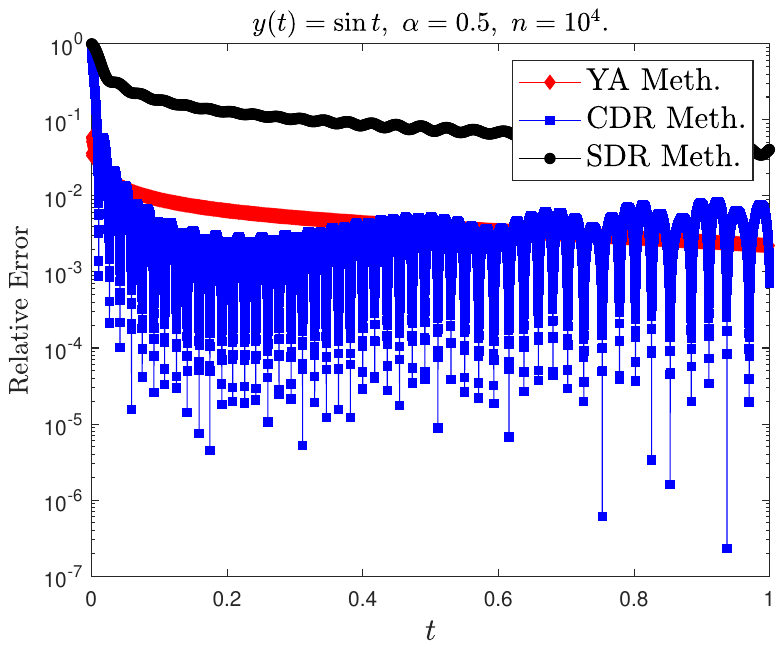}\\ \vspace{-5cm}
		\caption{The relative errors obtained by the backward Euler method for  three  methods YA, CDR and SDR for $n=10^3$ and $n=10^4$.}\label{Fig_7}
	\end{figure}
	\begin{figure}[H]
		\vspace{-4cm}
		\includegraphics[width=12cm]{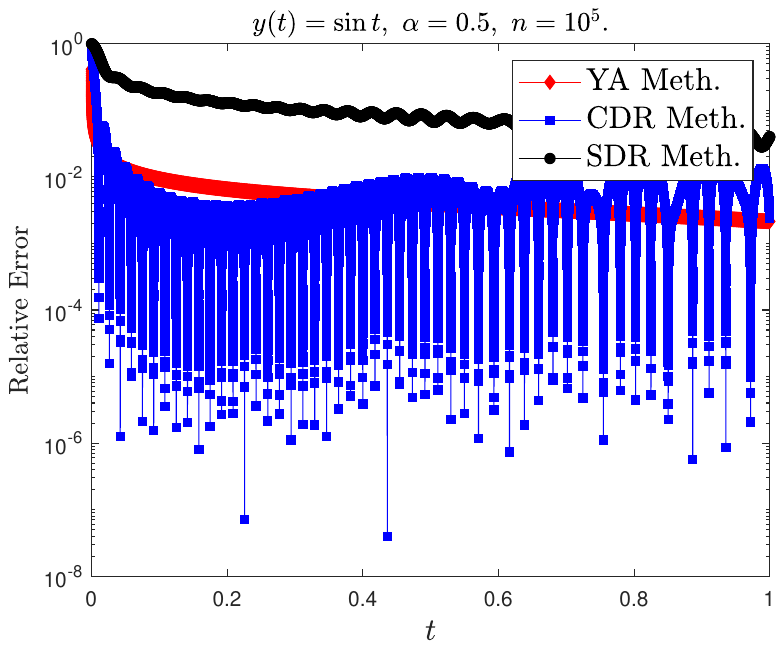}\hspace{-3cm}\includegraphics[width=12cm]{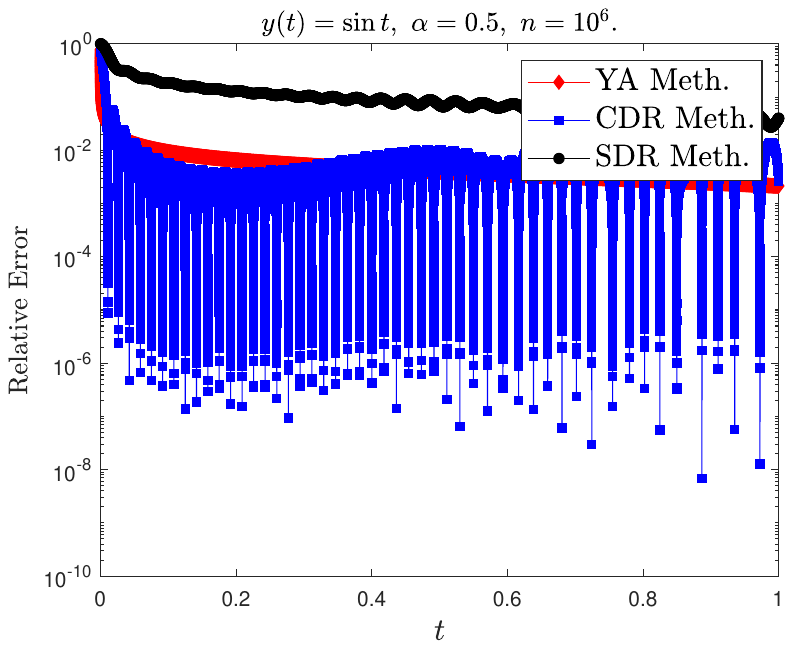}\\ \vspace{-5cm}
		\caption{The relative errors obtained by the backward Euler method for  three  methods YA, CDR and SDR for $n=10^5$ and $n=10^6$.}\label{Fig_8}
	\end{figure}
	\begin{figure}[H]
		\vspace{-4cm}
		\includegraphics[width=7cm,height=15cm]{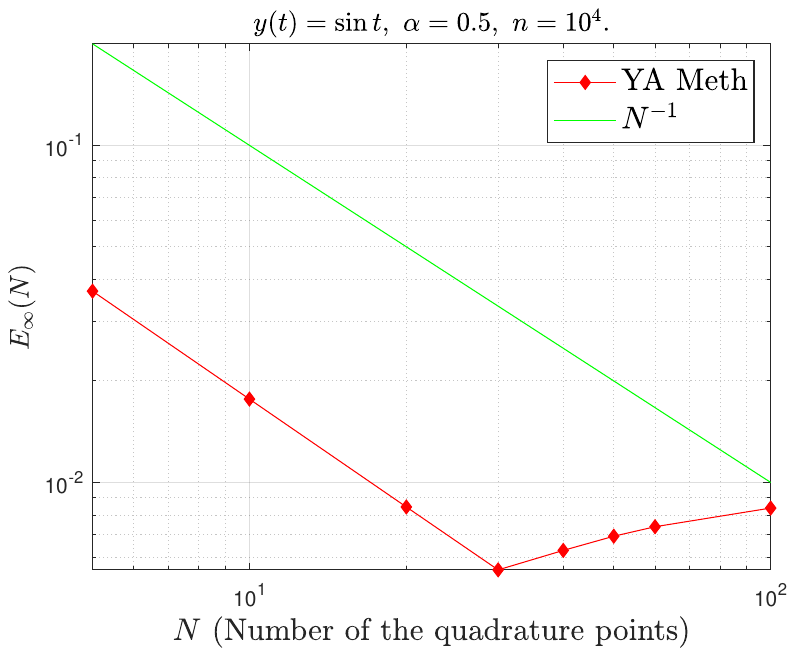}\hspace{-1.5cm}\includegraphics[width=7cm,height=15cm]{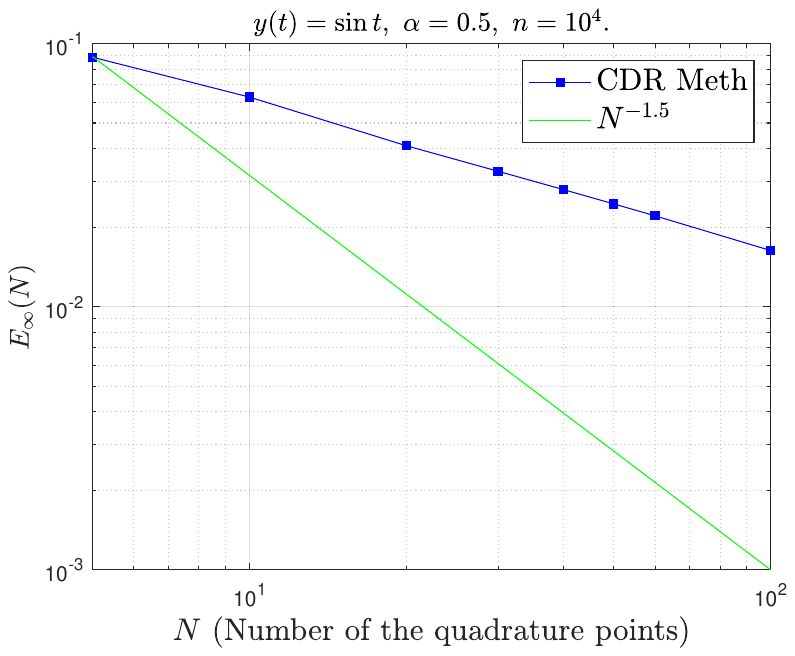}\hspace{-1.5cm}\includegraphics[width=7cm,height=15cm]{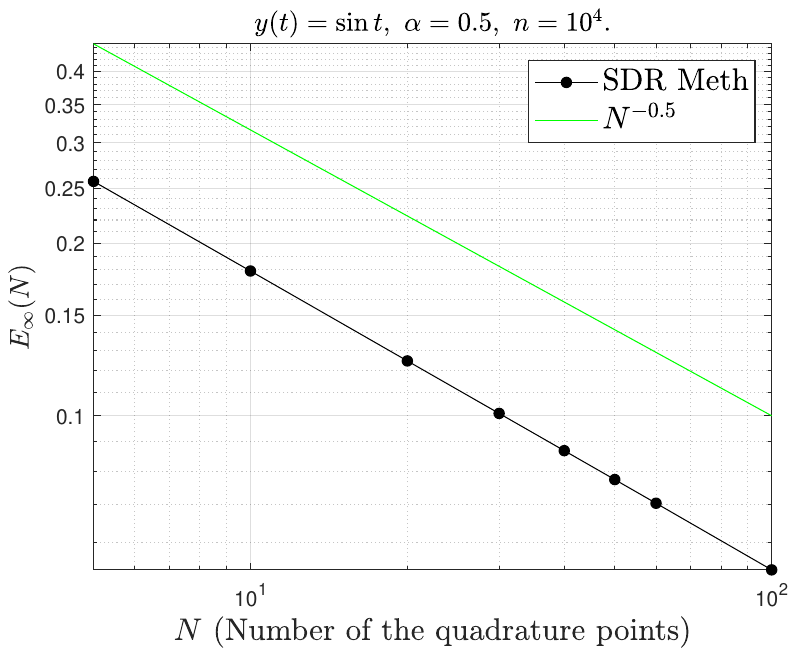}\\ \vspace{-2cm}
		\vspace{-2.5cm}
		\caption{Maximum norm of the errors obtained by the backward Euler method for three  methods YA, CDR and SDR with $n=10^4$ and various values of $N$.}\label{Fig_9}
	\end{figure}
\end{Example}
\begin{Example}\label{Ex_4}
	For the last example consider the function \cite{Sugiura2009}:
	\[
	y(t)=t^{\frac{\nu}{2}}J_\nu(2\sqrt{t}), \ t\in[0,1], \ \nu=3,
	\]	 
	where $J_\nu(z)$ is the Bessel function of the first kind. It is easy to show that:
	\[
	{}^{C}D_{0^+}^{\alpha}y(t)=t^{\frac{\nu-\alpha}{2}}J_{\nu-\alpha}(2\sqrt{t}),\ \alpha=0.5.
	\]
	Due to the fact that (cf. Property 2.3 of \cite{MR3742072}):
	\[
	J_\nu(t)\sim \frac{t^{\nu}}{2^\nu\Gamma(\nu+1)},\ t\to 0,\ \nu\ne-1,-2,\cdots.
	\]
	It is easy to verify that for $\nu=3$, we have $y\in C^3[0,1]$. The relative errors of the approximations of the Caputo fractional derivative of order $\alpha=0.5$ of the function $y(t)$ obtained from three mentioned methods are depicted in Figs. \ref{Fig_10} and \ref{Fig_11}. The maximum errors of the approximations for some values of $N$ with $n=10^4$ are also graphed in Fig. \ref{Fig_12}.
	\begin{figure}[H]
		\vspace{-5cm}
		\includegraphics[width=12cm]{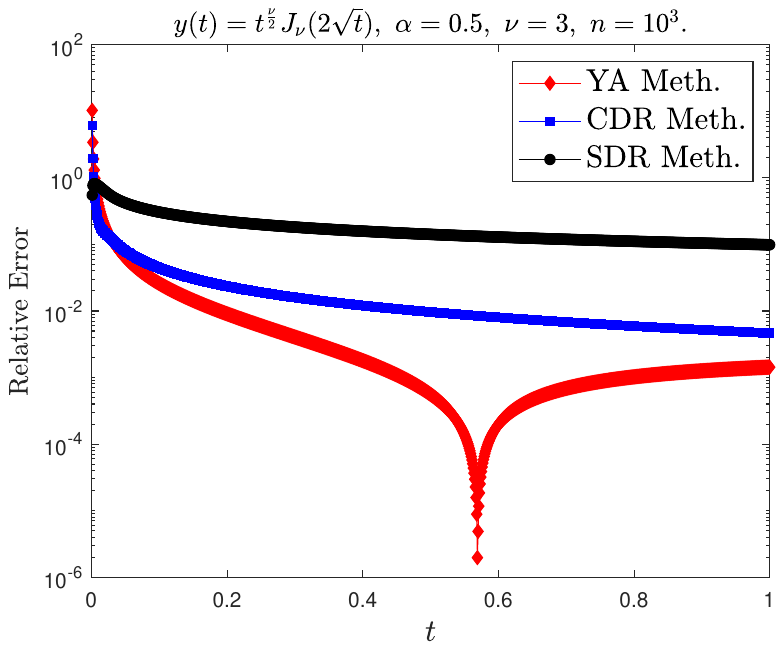}\hspace{-3cm}\includegraphics[width=12cm]{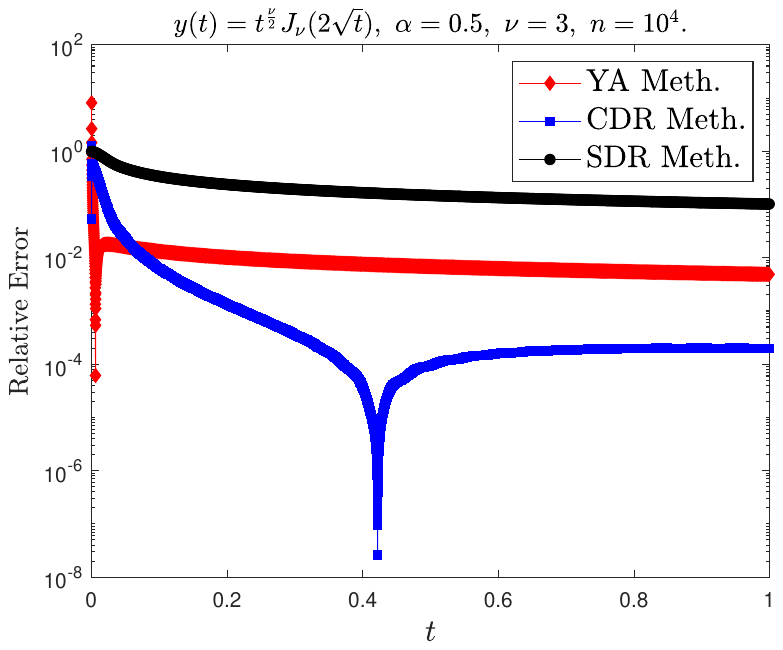}\\ \vspace{-5cm}
		\caption{The relative errors obtained by the backward Euler method for three  methods YA, CDR and SDR for $n=10^3$ and $n=10^4$.}\label{Fig_10}
	\end{figure}
	\begin{figure}[H]
		\vspace{-4cm}
		\includegraphics[width=12cm]{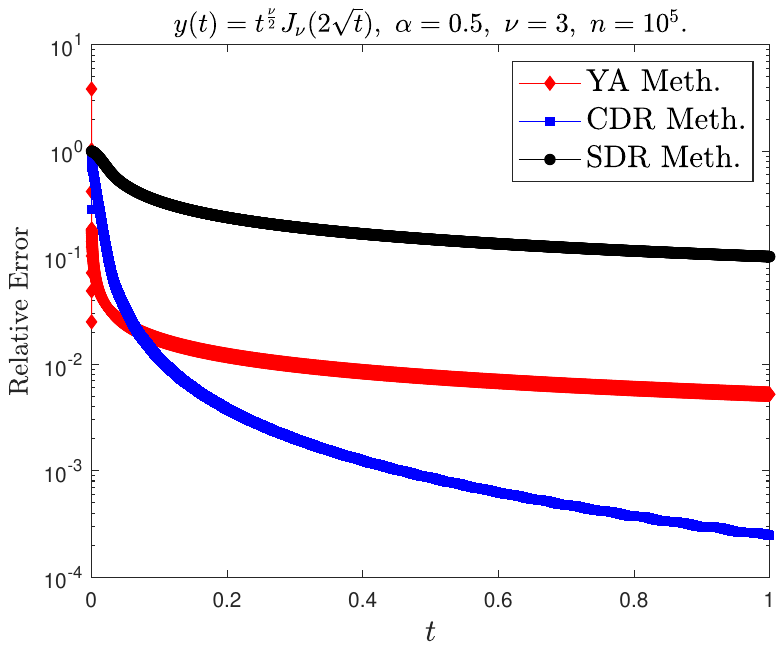}\hspace{-3cm}\includegraphics[width=12cm]{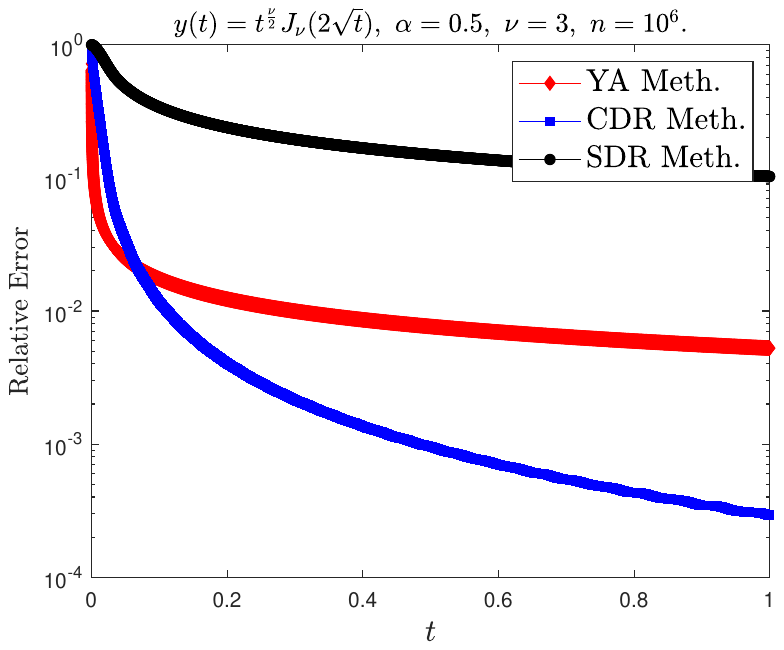}\\ \vspace{-5cm}
		\caption{The relative errors obtained by the backward Euler method for  three  methods YA, CDR and SDR for $n=10^5$ and $n=10^6$.}\label{Fig_11}
	\end{figure}
	\begin{figure}[H]
		\vspace{-4cm}
		\includegraphics[width=7cm,height=15cm]{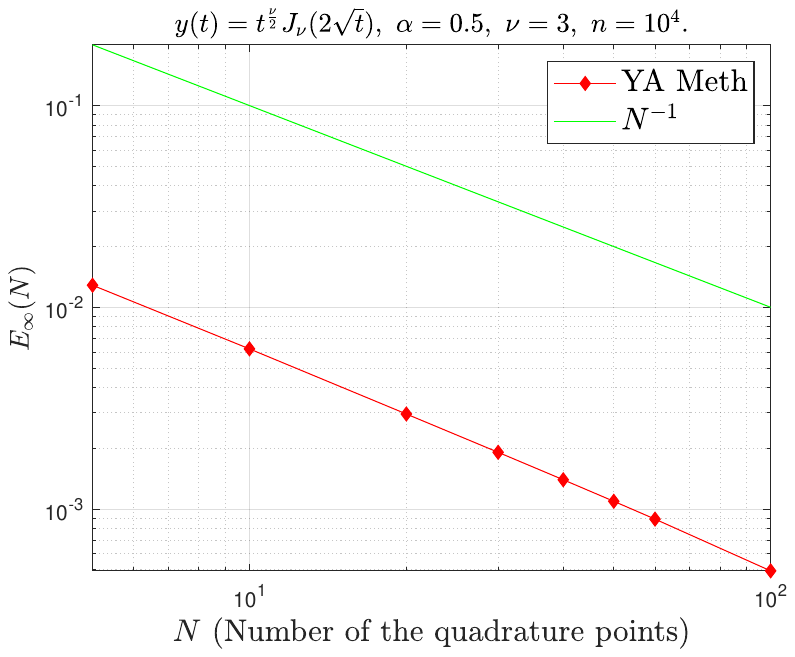}\hspace{-1.5cm}\includegraphics[width=7cm,height=15cm]{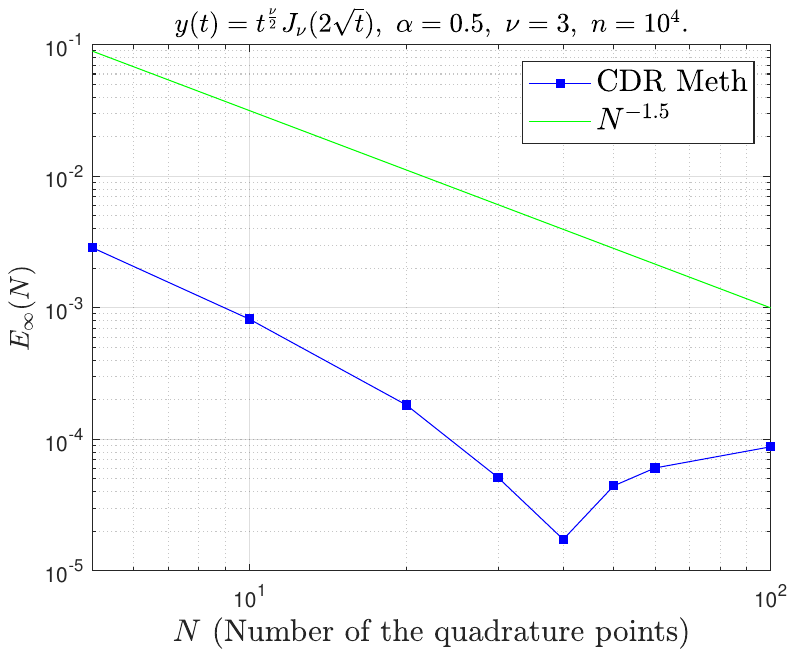}\hspace{-1.5cm}\includegraphics[width=7cm,height=15cm]{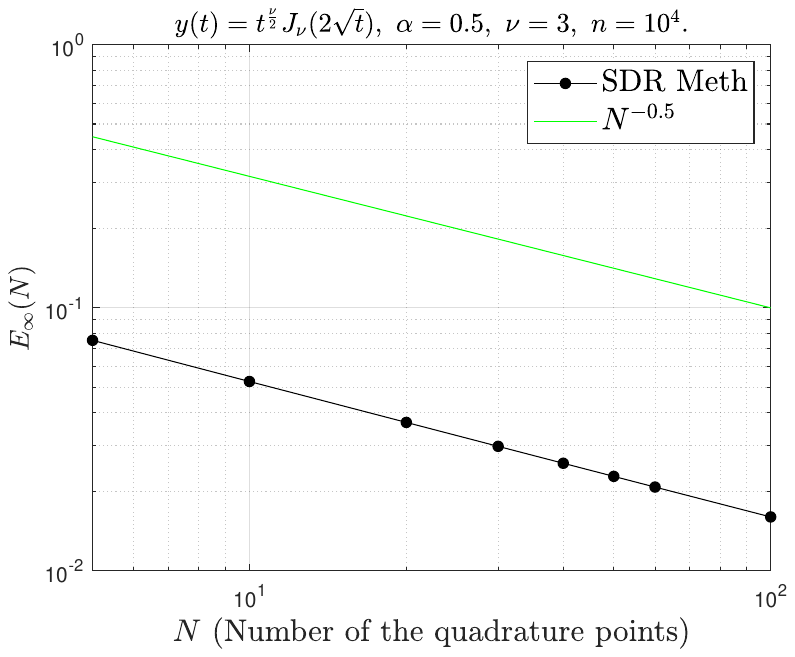}\\ \vspace{-2cm}
		\vspace{-2.5cm}
		\caption{Maximum norm of the errors obtained by the backward Euler method for  three  methods YA, CDR and SDR with $n=10^4$ and various values of $N$.}\label{Fig_12}
	\end{figure}
\end{Example}
\begin{Remark}\label{ImpRem}
Due to the results obtained from Examples \ref{Ex_1}-\ref{Ex_4}, the following conclusions can be drawn:
\begin{itemize}
	\item Numerical experiments show that  the first term of the overall errors of the methods CDR, SDR and YA proposed in Theorems \ref{OverErrCDRSDR} and \ref{OverErrYA} will usually dominate the overall error. This means that we can say that the errors of the methods CDR, SDR and YA decay like  $\mathcal{O}(N^{\alpha-2})$, $\mathcal{O}(N^{\alpha-1})$ and $\mathcal{O}(N^{2\alpha-2})$, respectively, where $N$ is the number of quadrature points. So, we don't need to apply these methods with small values of $h$. In other word, it is not necessary to use the step size which satisfies $h<N^{-2}$.
	\item As we expected from Theorems \ref{OverErrCDRSDR} and \ref{OverErrYA},  the convergence rate of the SDR method is (very) slow while the error of the CDR method (for function $y(t)$ for which $y''\in C[a,b]$) decays faster than the YA method. 
\end{itemize} 
\end{Remark}
\subsection{An improvement of SDR method }\label{Sec_3.1}
As we saw in the previous section, the convergence rate of the SDR method is very slow. In fact, the slow convergence of the SDR method comes from the asymptotic behavior of $z^{\alpha}\omega^S(z,t)$ when $z\to0$ and $z\to+\infty$. As it can be seen in Theorems \ref{AsymtoticCDRandSDR} and \ref{ErrorQuad}, the convergence rate of the $N$-point generalized Gauss-Laguerre formula for the SDR method applied to Examples \ref{Ex_1} and \ref{Ex_2} is proportional to $\mathcal{O}\left(N^{-0.6}\right)$  and $\mathcal{O}\left(N^{-0.4}\right)$, respectively.

To improve this difficulty, we will use a simple change of variable
$z=\theta^2$ in Theorem \ref{Sin}. So, we have the following theorem:
\begin{Theorem}\label{ImSin}
	(The improved sine diffusive representation (ISDR)). For $0<\alpha<1$, we have
	\begin{equation*}\label{SDR1}
		{}^{C}D_{0^+}^{\alpha}y(t)=\frac{4\cos(\tfrac{\pi\alpha}{2})}{\pi}\int_{0}^{\infty}\theta^{2\alpha-1}\left(\int_{0}^{t}\sin\left((t-\tau)\theta^2\right){y'(\tau)}\,d\tau\right)\,d\tau=\int_{0}^{\infty}\theta^{2\alpha-1}\omega^{IS}(\theta,t)\,d\theta,
	\end{equation*}
	where
	\begin{equation}\label{Kernel_2_Imp}
		\omega^{IS}(\theta,t)=\frac{4\cos(\tfrac{\pi\alpha}{2})}{\pi}  \left(\int_{0}^{t}\sin\left((t-\tau)\theta^2\right){y'(\tau)}\,d\tau\right).
	\end{equation}
	Also, for a given function $y$ for which its first derivative exists on $[0,T]$, $\omega^{IS}(\theta,t)$ (for fixed $\theta>0$) satisfies the following second-order differential equation:
	\begin{equation}\label{ODE3IM}
		\begin{cases}
			\displaystyle\frac{\partial^2 \omega^{IS}}{\partial t^2}+\theta^4\omega^{IS}=\frac{4\cos(\tfrac{\pi\alpha}{2})}{\pi}\ \theta^2y'(t),\\ 	\displaystyle \omega^{IS}(\theta,0)=\frac{\partial}{\partial t}\omega^{IS}(\theta,0)=0.
		\end{cases}
	\end{equation}
	\begin{proof}
		The proof is straightforward. 
	\end{proof}
\end{Theorem}
In the following we present the error analysis of the new improvement of the SDR method which we denote by ISDR. The error analysis of the ISDR is fairly similar to those provided in Section \ref{Sec_2.2}. So we denote:
\begin{eqnarray}\label{Resid_ISDR}
	R^{\alpha}_{IS,N,h}y(t)&:=&{}^{C}D_{0^+}^{\alpha}y(t)-{}^{C}D_{0^+,IS,N,h}^{\alpha}y(t)\nonumber\\
	&=&\int_{0}^{+\infty}z^{2\alpha-1}\omega^{IS}(z,t)\,dz-\sum_{i=1}^{N}w_i^{(2\alpha-1)}\ e^{z_i^{(2\alpha-1)}}\omega^{IS}_h(z_i^{(\alpha)},t)\nonumber\\
	&=&R_{IS,N,\alpha}^{GGL}\left[ e^{(.)}\omega^{IS}(.,t)\right]+Q_{IS,N,\alpha}^{GGL}\left[ e^{(.)}E^{IS}_h(.,t)\right], 
\end{eqnarray}
where 
\begin{equation}\label{ErrQuadISDR}
	R_{IS,N,\alpha}^{GGL}\left[ e^{(.)}\omega^{IS}(.,t)\right]:=\int_{0}^{+\infty}z^{2\alpha-1}\omega^{IS}(z,t)\,dz-\sum_{i=1}^{N}w_i^{(2\alpha-1)}\ e^{z_i^{(2\alpha-1)}}\omega^{IS}(z_i^{(2\alpha-1)},t),
\end{equation}
denotes the error of generalized Gauss-Laguerre formula with respect to the weight function $w(z)=z^{2\alpha-1}e^{-z}$ and
\begin{equation}\label{ErrODEISDR}
	Q_{IS,N,\alpha}^{GGL}\left[ e^{(.)}E^{IS}_h(.,t)\right]:=\sum_{i=1}^{N}w_i^{(2\alpha-1)}\ e^{z_i^{(2\alpha-1)}}\left[\omega^{IS}(z_i^{(2\alpha-1)},t)-\omega^{IS}_h(z_i^{(2\alpha-1)},t)\right],
\end{equation}
used for the error of the ODE solver. In the next theorem the asymptotic behavior of the function $\omega^{IS}(z,t)$ when $z\to 0$ and $z\to+\infty$  is provided. 
\begin{Theorem}
		Let $t \in(0, T)$ be fixed and $0<\alpha<1$.
	\begin{enumerate}
		\item[(a)] Assume that there exists some constant $C>0$, such that for all $t\in(0,T)$ we have  $|y'(t)|>C$ then function  $\omega^{IS}(.,t)$ defined  in \eqref{Kernel_2_Imp} behave as:
		\begin{equation}\label{Asymp_1-imp}
			z^{2\alpha-1}\omega^{IS}(z,t)\sim z^{2\alpha+1}\ \ \ \text{as}\ z\to 0.
		\end{equation}
		\item[(b)] Let $y(t)\in C^1[0,T]$ and $y(t)$ and $y'(t)$ be  of exponential order then we have:
		\begin{eqnarray}\label{Asymp_3_imp}
			&& z^{2\alpha-1}\omega^{IS}(z,t)\sim z^{2\alpha-3}\ \ \ \text{as}\ z\to +\infty.
		\end{eqnarray}
	\end{enumerate}
\end{Theorem}
 \begin{proof}
 The proof is fairly similar to the proof of Theorem \ref{AsymtoticCDRandSDR}. 
 \end{proof}
The next theorem, gives the error bound of the new improvement of the SDR method.
\begin{Theorem}\label{OverErrISDR}
	Let $0<\alpha<1$. If a A-stable one-step implicit method of order $p$ with the step size $h<N^{-4}$, (where $N$ is the number of integration points in the generalized Gauss-Laguerre formula) is used for  Eq. \eqref{ODE3IM}, then 	for $y(t)\in C^1[0,T]$ and $t\in[0,T]$, we have the overall error analysis of ISDR approximation formula:

		\begin{equation}\label{OVERERRISDR}
			\Big|R^{\alpha}_{IS,N,h}y(t)\Big|=\mathcal{O}(N^{2\alpha-2})+\mathcal{O}(h^p)\int_{0}^{4N}e^{3T\displaystyle z^4 }\,dz.
		\end{equation}
	
\end{Theorem}
\begin{proof}
The proof is obtained by the similar fashion which used  for  	Theorem \ref{OverErrCDRSDR}.
\end{proof}
\begin{Remark}
	It is worthy to point out that as we stated in Remark \ref{OverErrCDRSDR}, in practice,  we don't need to use the step size $h$ in such a way $h<N^{-4}$.  
\end{Remark}
\begin{Remark}
Thanks to the overall error of the ISDR presented in Theorem \ref{OverErrISDR}, 	we expect that the new improvement method can work like  the YA one (See Theorem \ref{OverErrYA}).  
\end{Remark}
\begin{Example}
To show the efficiency and accuracy of the ISDR, we use this method to approximate the Caputo fractional derivative of order $\alpha$ of the following functions:
\begin{eqnarray*}
	&& y(t)=t^{1.6},\ \alpha=0.4,\ t\in[0,3],\\
	&& y(t)=t^{3},\ \alpha=0.6,\ t\in[0,1],\\
    && y(t)=\sin t,\ \alpha=0.5,\ t\in[0,1],\\
    && y(t)=t^{\frac{\nu}{2}}J_{\nu}(2\sqrt{t}),\ \nu=3,\ \alpha=0.5,\ t\in[0,1].
\end{eqnarray*}

To make a good comparison,  the maximal errors of the approximation methods CDR, SDR, YA and ISDR obtained by the backward Euler method with $n=10^4$ and various values of $N$ of these functions are shown in Figs. \ref{Fig_14} and \ref{Fig_15}.

It is clearly observed from  Figs. \ref{Fig_14} and \ref{Fig_15} that, although,  the convergence rate of the ISDR method applied to the functions $y(t)=t^3$ and $y(t)=t^{\frac{\nu}{2}}J_{\nu}(2\sqrt{t})$ with $\nu=3$, is the same as the YA method, but for other functions, the errors of ISDR method decay like the SDR one and thus we have the surprising results. This problem (may) comes from the smoothness of the Caputo fractional derivative of the functions $y(t)=t^{1.6}$ and $y(t)=\sin t$. 
\begin{figure}[H]
	\includegraphics[width=12cm]{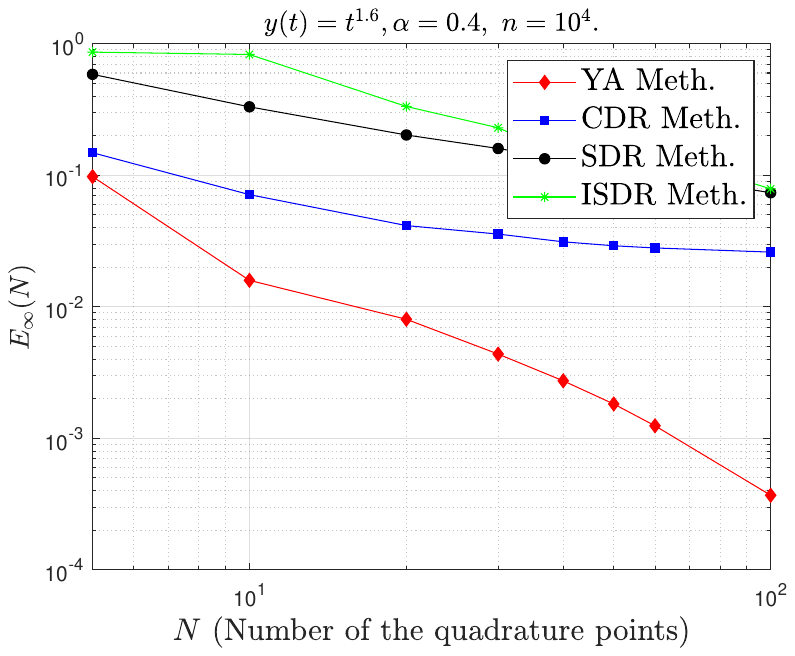}
	\hspace{-3cm}\includegraphics[width=12cm]{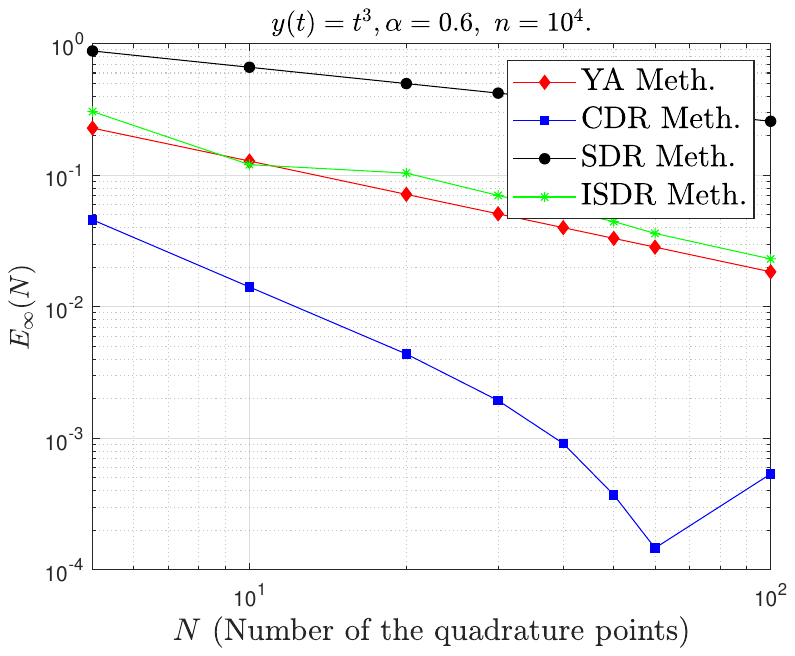}\vspace{-5cm}
	\caption{Maximum norm of the errors obtained by the backward Euler method for four  methods YA, CDR, SDR and ISDR with $n=10^4$ and various values of $N$.}\label{Fig_14}
\end{figure}
\begin{figure}[H]
	\vspace{-4cm}
	\includegraphics[width=12cm]{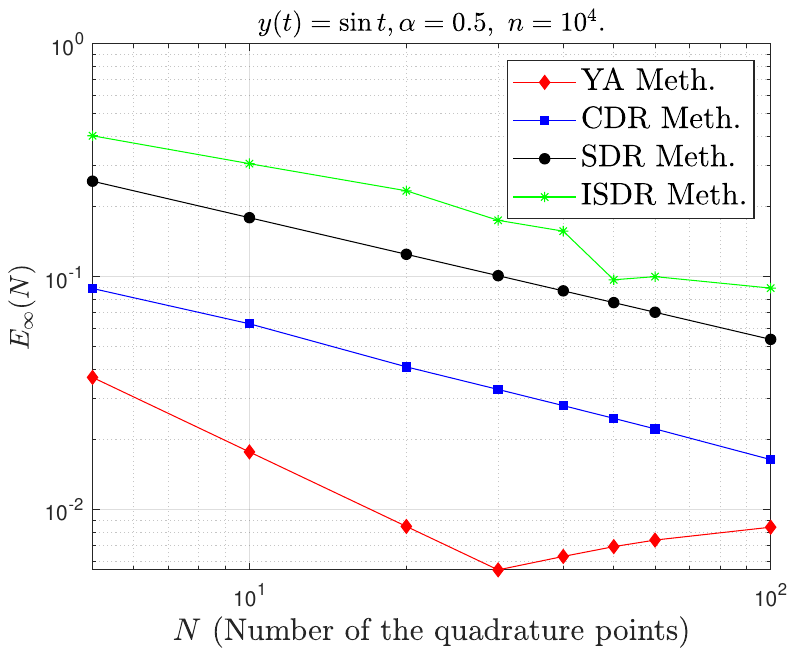}\hspace{-3cm}\includegraphics[width=12cm]{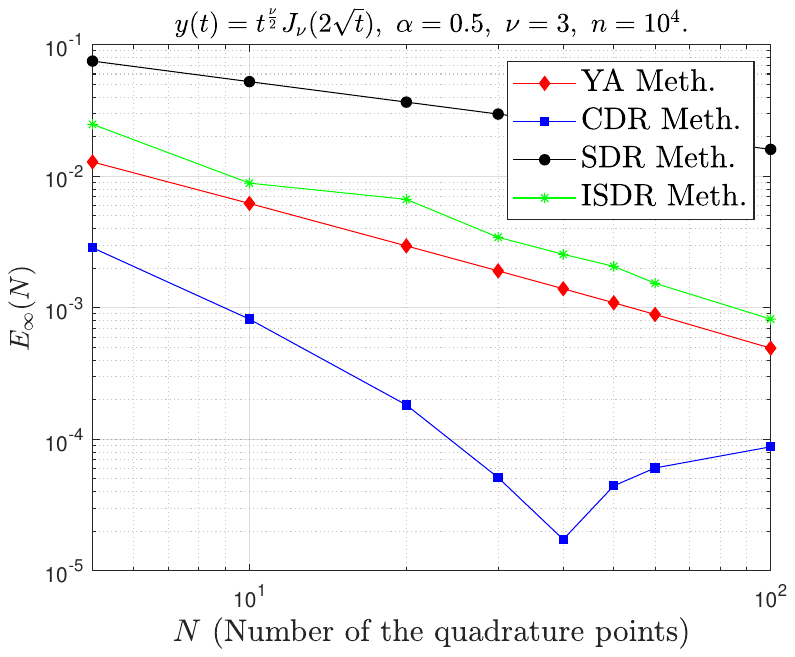}\\ \vspace{-2cm}
	\vspace{-2.5cm}
	\caption{Maximum norm of the errors obtained by the backward Euler method for  four  methods YA, CDR, SDR and ISDR with $n=10^4$ and various values of $N$.}\label{Fig_15}
\end{figure}
\end{Example}
\section{Concluding remarks and future works  }\label{Sec_4}
The diffusive representation for the Caputo fractional derivative has very interesting feature from the numerical point of view (See \cite{Diethelm2009,Diethelm2008}). This paper presents two new classes of diffusive representations with sine and cosine kernels to approximate the Caputo fractional derivative which were called as the cosine and sine diffusive representations and denoted by CDR and SDR, respectively. The error analysis of the CDR and  SDR methods proved in detail. 

Some numerical examples have also provided to show the efficiency and accuracy of the new methods. Our numerical experiments show that for function $y(t)$ which ${}^{C}D_{0^+}^{\alpha}y(t)\in C^2[0,T]$, in opposite to the SDR method, the CDR one is faster than the Yuan and Agrawal (YA) method.  So, in the final part of the paper, a new version of the SDR method (which was denoted by ISDR) is also proposed and verified numerically. The maximal error of the ISDR method for function $y(t)$ which ${}^{C}D_{0^+}^{\alpha}y(t)\in C^2[0,T]$ decays like  YA method. 

The authors believed that the proposed methods will open a new window for researchers to investigate the diffusive representation methods in more and in-depth details. So, in the following, the authors suggest some future works which can be considered in the continuation of this paper. 
\begin{enumerate}
	\item The first suggestion is to propose some new modifications and improvements of the CDR and SDR methods to obtain some fast and accurate numerical methods to approximate the Caputo fractional derivative (See some improvements of the YA method by Diethelm and et al., \cite{Diethelm2008,Diethelm2009,Diethelm2021,AnewDiff,Liu2018,MR3936246,Birk2010}).
	\item As we saw in the previous sections, the CDR and SDR methods have been used to approximate the Caputo fractional derivative of order $\alpha\in(0,1)$. So, our second suggestion is to extend these methods for $\alpha>1$. 
	\item Our third suggestion is to apply the CDR and SDR methods to solve problem with fractional derivatives such as:
	\begin{enumerate}
		\item  Fractional ordinary and partial differential equation.
		\item Fractional optimal control and calculus of variation problems.
	\end{enumerate}
	\item Our last suggestion is to follow the idea of this paper to introduce some new generalizations of the diffusive representation.
\end{enumerate}
\section{Acknowledgment}
The authors would like to express their special thanks to Professor K. Diethelm for his helpful comments and suggestions on the first version of the current paper.
%

\end{document}